\documentclass{amsart}
\usepackage[utf8]{inputenc}
\usepackage{microtype}
\usepackage{amsmath}
\usepackage{amssymb}
\usepackage{amsthm}
\usepackage{amsfonts}
\usepackage{booktabs}
\usepackage[final]{graphicx}
\usepackage{setspace}
\usepackage[center]{caption}
\usepackage{tikz}
\usetikzlibrary{shapes}
\usepackage{tikz-cd}
\usepackage[active]{srcltx}
\usepackage{mathtools}
\usepackage{mathrsfs}

\usepackage[colorlinks=true, allcolors=blue]{hyperref}

\usepackage[top=3cm,bottom=2cm,left=3cm,right=3cm,marginparwidth=1.75cm]{geometry}
\usepackage{bm}
\usepackage{dsfont}
\usepackage{enumitem}
\usepackage{relsize}
\theoremstyle{plain}\newtheorem{theorem}{Theorem}[]
\theoremstyle{plain}\newtheorem*{theorem*}{Theorem}
\theoremstyle{plain}\newtheorem{proposition}[theorem]{Proposition}
\theoremstyle{plain}\newtheorem{lemma}[theorem]{Lemma}
\theoremstyle{plain}\newtheorem{corollary}[theorem]{Corollary}
\theoremstyle{plain}\newtheorem*{corollary*}{Corollary}
\theoremstyle{definition}\newtheorem{definition}[theorem]{Definition}
\theoremstyle{definition}\newtheorem{remark}[theorem]{Remark}
\theoremstyle{definition}\newtheorem{example}[theorem]{Example}

\theoremstyle{thm}
\newtheorem{thm}{Theorem}
\newtheorem{cor}[thm]{Corollary}

\DeclarePairedDelimiter{\abs}{\lvert}{\rvert}
\DeclarePairedDelimiter{\norm}{\|}{\|}
\newcommand{\R}{\mathds{R}}
\newcommand{\C}{\mathds{C}}
\newcommand{\N}{\mathds{N}}
\newcommand{\Z}{\mathds{Z}}

\newcommand{\PV}{\mathrm{P}V}
\newcommand{\SV}{\Sigma}
\newcommand{\SW}{V}
\newcommand{\W}{\mathcal{V}}
\newcommand{\tr}{\mathrm{tr}}

\title{What is the probability that a random symmetric tensor is close to rank-one?}
\author{Alberto Cazzaniga, Antonio Lerario, Andrea Rosana}

\begin{document}

\maketitle

\begin{abstract}We address the general problem of estimating the probability that a real symmetric tensor is close to rank--one tensors. Using Weyl's tube formula, we turn this question into a differential geometric one involving the study of metric invariants of the real Veronese variety. 
More precisely, we give an explicit formula for its  reach  and curvature coefficients with respect to the Bombieri--Weyl metric. These results are obtained using techniques from Random Matrix theory and an explicit description of the second fundamental form of the Veronese variety in terms of GOE matrices. Our findings give a complete solution to the original problem. In the case of rational normal curves it leads to a simple formula describing explicitly exponential decay with respect to the degree of the tensor.

\end{abstract}

\section{Introduction}
\subsection{What is the probability that a random symmetric tensor is close to rank-one?}
Over the last decades, symmetric tensors have been proven to be a very flexible and valuable tool in many different contexts. In particular, rank--one approximation and tensor decomposition found applications in machine learning (\cite{tensorslatentvariable}), signal processing and image analysis (\cite{tensorsignal}, \cite[Ch.3, 4]{Sakatasurvey}), chemistry (\cite{tensorschemical}), statistics (\cite{McCullagh}), psychology and medical diagnostics (\cite{Kroonenberg, diffusionbrain}), phylogenetics (\cite[Ch.5]{Sakatasurvey}, \cite{Landsbergtensors}), and quantum computing (\cite{mostquantum}), to name a few. Motivated by this, in this paper we address the following question: 
\[\textit{\lq\lq What is the probability for a real symmetric tensor to be \lq\lq close\rq\rq \ to rank--one?\rq\rq}\] 
To make this question more precise, we need to introduce a natural notion of distance and a natural probability measure on the space $\mathcal{S}^d(n)$ of real symmetric tensors on $\R^{n}$ of order $d$. Recall first that this is the subspace of  $(\R^n)^{\otimes d}$ consisting of elements which are invariant under the action of the symmetric group under permutation of the factors. The set $\mathcal{S}_1^d(n)$ of symmetric, \emph{rank--one} tensors consists of the set of elements $T\in \mathcal{S}^d(n)$ of the form $T=\pm v\otimes \cdots \otimes v$ for some $v\in \R^n$.

The space of all tensors is endowed with the \emph{Frobenius scalar product}, denoted by $\langle \cdot,\cdot\rangle_{F}$ and defined as follows. Letting $\{e_1, \ldots, e_n\}$ be the standard basis of $\R^n$, then a basis for the space of tensors is given by $\{e_{i_1}\otimes\cdots\otimes e_{i_d}\,|\,1\leq  i_1, \ldots, i_d\leq n\}$  and the Frobenius scalar product is obtained by declaring this basis to be orthonormal. We still denote by $\langle\cdot, \cdot\rangle_F$ the restriction of the Frobenius scalar product to the space of symmetric tensors and we use the notation $\|\cdot\|_F$ for the associated norm and $\mathrm{dist}_F(\cdot, \cdot)$ for the induced distance function.

The choice of a scalar product on a finite-dimensional vector space naturally leads to the definition of a Gaussian measure on it (this comes for free without introducing further structure). In the case of our interest, the space of symmetric tensors with the Frobenius scalar product can be turned into a \emph{Gaussian probability space} by defining for every Borel set $U\subseteq {\mathcal{S}^d(n)}$ 
\begin{equation}\label{eq:gaussiantensor}\mathbb{P}\bigg\{T\in U\bigg\}:=\frac{\displaystyle \int_U \mathrm{exp}\left({-\frac{1}{2}\|T\|_F^2}\right)\,\mathrm{d\mu}}{\displaystyle \int_{{\mathcal{S}^d(n)}} \exp\left({-\frac{1}{2}\|T\|_F^2}\right)\,\mathrm{d\mu}},
\end{equation}
where $\mathrm{d}\mu$ denotes the integration with respect to the Lebesgue measure on ${\mathcal{S}^d(n)}$. Notice that the Frobenius scalar product has a natural invariance under the action of the orthogonal group on the space of tensors and the resulting Gaussian probability distribution inherits this invariance (see Section \ref{GOEKostlansection}).

With this notation, the above question can be phrased more precisely as follows: we are required to compute, for a given $\delta>0$, the quantity: 
\begin{equation}\label{eq:proba1}\mathbb{P}\left\{T\in {\mathcal{S}^d(n)}\,\bigg|\, \mathrm{dist}_F(T, \mathcal{S}_1^d(n))\leq \delta \|T\|_{\mathrm{F}}\right\}.\end{equation}
Notice that we have turned the question into a conic problem that takes into account also the norm of the tensor, as it is common procedure in numerical algebraic geometry \cite{Burgissercondition}.

Before proceeding, we discuss two examples of applications of the results of the current paper, more precisely of Theorem \ref{thm:volumeintro}, which gives a closed formula for the probability in \eqref{eq:proba1}.

\begin{example}[Real symmetric matrices]\label{example1}When $d=2$ the space $\mathcal{S}^2(n)$ can be naturally identified with the space of $n\times n$ real symmetric matrices and the Frobenius scalar product is given by 
 $$\langle A, B\rangle_{\mathrm{F}}=\tr(AB).$$ 
In this case, by the Spectral Theorem, the set $\mathcal{S}^2_1(n)$ of rank--one tensors coincides with the set of matrices of the form $A=\pm vv^t$, for some $v\in \R^n$.
The space of symmetric matrices together with the Gaussian distribution \eqref{eq:gaussiantensor} is called by probabilists the \emph{Gaussian Orthogonal Ensemble}, and denoted by $\mathrm{GOE}(n)$ (see Section \ref{GOEKostlansection}). Ordering the singular values of $Q$ in increasing order, $\sigma_1(Q)\leq\cdots\leq \sigma_n(Q)$, by the Eckart--Young Theorem, the distance in the Frobenius norm between $Q$ and $\mathcal{S}^2_1(n)$ is given by
$$\mathrm{dist}_{F}(Q, \mathcal{S}^2_1(n))=\left(\sum_{k=1}^{n-1}\sigma_k(Q)^2\right)^{\frac{1}{2}}.$$ 
In this case, using the fact that $\|Q\|_F^2=\sigma_1(Q)^2+\cdots+\sigma_n(Q)^2$, the probability in \eqref{eq:proba1} equals
\begin{equation}\label{eq:proba2}\mathbb{P}\bigg\{(1-\delta^2)\|Q\|_F^2\leq \sigma_n(Q)^2\bigg\},\end{equation}
therefore our question becomes of interest from the point of view of Random Matrix Theory. In this case, we can apply Theorem \ref{thm:volumeintro} and obtain a closed formula for the probability in \eqref{eq:proba2}, for $\delta\leq 1/\sqrt{2}$. For instance:
\begin{align}\label{eq:proba3}\mathbb{P}\bigg\{Q\in \mathcal{S}^2(3)\,\bigg|\,(1-\delta^2)\|Q\|_F^2\leq \sigma_3(Q)^2\bigg\}&=\frac{1}{3 \pi}\left(2 \delta \sqrt{1 - \delta^2} (-3 + 14 \delta^2) + 
 6 \arctan\left(\frac{\delta}{ \sqrt{1 - 
 \delta^2}}\right)\right)\\
 &=\frac{32 \delta^3}{3 \pi} - \frac{64 \delta^5}{15\pi} + \mathcal{O}(\delta^6).
 \end{align}
 In the context of Random Matrix Theory, which is concerned with the study of the distribution of the eigenvalues of random symmetric matrices, expressions like \eqref{eq:proba3} are especially interesting. In fact, most of the theory is focused on asymptotic results with the size of the matrices that goes to infinity (see, for instance, \cite[Chapter 3]{Zeitounirandom}) and very little is known on the distribution of the eigenvalues of random matrices of  fixed (small) size. 
\end{example}

\begin{example}[Geometric measure of entanglement]
Quantum entanglement is one of the main topics of research in quantum information theory, where it provides an important resource for quantum computing. Recall that a $d$-partite pure (symmetric) state of a quantum system can be regarded as a normalized (symmetric) tensor in the tensor product of $d$ Hilbert spaces $\mathcal{H}=\otimes_{k=1}^d \mathcal{H}_k$. A state $|\phi\rangle \in \mathcal{H}$ is called \emph{separable} if it is a product state $|\phi\rangle = \otimes_{k=1}^d|\phi^{(k)}\rangle$ with $\||\phi^{(k)}\rangle\|_{\mathcal{H}_k}=1$ for $k=1,\dots,d$; it is called \emph{entangled} if it is not separable. In the language of tensors, separable pure states correspond to rank-one tensors, while entangled states to tensors of higher rank. \\ A way to quantify the entaglement of a pure state, i.e. to measure how distant that state is from being separable, is the \emph{geometric measure of entanglement}, first introduced in \cite{Shimony} (see also \cite{entanglementsymmetric} for symmetric states). For a pure symmetric state $|\psi\rangle$ this is defined as \cite{HuQiZhang}
\begin{align}\label{generalentanglement}
    E_G(|\psi\rangle):= 1 - \underset{\Phi=|\phi\rangle^{\otimes d} \in \mathcal{H}, \||\phi\rangle\|=1}{\mathrm{max}} | \langle \psi | \Phi \rangle |.
\end{align}
When the space is $\mathcal{H}=(\R^n)^{\otimes d}$ endowed with the Euclidean norm, pure symmetric states are represented by a tensor $T \in S^d(n)$ with Frobenius norm $1$ and \eqref{generalentanglement} becomes
\begin{align}\label{symmetricentanglement}
    E_G(T)=1-\underset{x \in S^{n-1}}{\mathrm{max}} | \langle T,x^{\otimes d}\rangle_F |.
\end{align}
In the tensor community, given $T \in S^d(n)$ and $x \in S^{n-1}$ the quantity 
\begin{align*}
    Tx^d:= \langle T,x^{\otimes d}\rangle_F
\end{align*}
is known as the \textit{Generalized Rayleigh Quotient (GQR)} \cite{ZhangGRQ}. The problem of maximizing the absolute value of the GRQ over $S^{n-1}$ for a tensor $T$ is equivalent to finding the best rank-one approximation of $T$. Indeed, given $x_* \in S^{n-1}$ such that 
\begin{align}
    | Tx_*^d |= \underset{x \in S^{n-1}}{\mathrm{max}}| Tx^d |,
\end{align}
we have that $(Tx_*^d) x_*^{\otimes d}$ is the best rank-one approximation of $T$ and 
\begin{align}\label{distanceGRQ}
    \mathrm{dist}_F^2(T,S^d_1(n))= \|T\|_F^2 - (Tx_*^d)^2,
\end{align}
see \cite[Theorem 2.19]{Qitensoranalysis}. When $T$ represents a pure symmetric state as above, by definition $\|T\|_F=1$ and \eqref{symmetricentanglement} becomes
\begin{align}
    E_G(T)= 1 - \sqrt{1-\mathrm{dist}_F^2(T,S_1^d(n))}.
\end{align}
A bound on the distance of $T$ from rank-one symmetric tensors implies a bound on the geometric measure of entanglement of the pure symmetric state represented by $T$. In particular, we see that 
\begin{align*}
\mathbb{P}\left\{T\in {\mathcal{S}^d(n)},\, \| T \|_F=1 \, \bigg|\, \mathrm{dist}_F(T, \mathcal{S}_1^d(n))\leq \delta \right\}=\mathbb{P}\left\{E_G(T) \leq 1-\sqrt{1-\delta^2}\right\},
\end{align*}
where the left-hand side is equivalent to \eqref{eq:proba1}. \\ An alternative formulation of the geometric measure of entanglement sometimes used in quantum computing \cite{mostquantum} is given as 
\begin{align*}
    \tilde E_G(|\psi\rangle) := -\mathrm{log}_2 \left(\underset{\Phi=|\phi\rangle^{\otimes d} \in \mathcal{H}, \||\phi\rangle\|=1}{\mathrm{max}}|\langle \psi | \Phi \rangle |^2\right).
\end{align*}
With the same reasoning as before we obtain
\begin{align}\label{alternativeprobability}
\mathbb{P}\left\{T\in {\mathcal{S}^d(n)},\, \| T \|_F=1 \, \bigg|\, \mathrm{dist}_F(T, \mathcal{S}_1^d(n))\leq \delta \right\}=\mathbb{P}\left\{\tilde E_G(T) \leq -\mathrm{log}_2(1-\delta^2)\right\}.
\end{align}
We refer the reader to  \cite{mostquantum} for a detailed discussion on how the knowledge of bounds such as \eqref{alternativeprobability} might be relevant for quantum computing techniques and algorithms.
\end{example}

\subsection{A geometric formulation of the problem using random polynomials}It will be useful for us to identify symmetric tensors with homogeneous polynomials (and, for notational reasons, to work with tensors on $\R^{n+1}$ rather than on $\R^n$). More precisely, to every $T\in \mathcal{S}^d(n+1)$ we associate the homogeneous polynomial $p_T:\R^{n+1}\to \R$ defined by
\begin{equation}
p_T(x):=\langle T, x^{\otimes d}\rangle_F.
\end{equation}
This gives a linear isomorphism
\begin{equation}\label{eq:phi}\phi:\mathcal{S}^d(n+1)\stackrel{\simeq}{\longrightarrow} \R[x_0, \ldots, x_n]_{(d)},\end{equation}
which, when $d=2$, is the familiar identification between a symmetric matrix $Q$ and the associated quadratic form $p_Q(x)=\langle x,Qx\rangle.$ The advantage of this approach is that it connects to well--studied objects in the theory of random polynomials and it allows to give a neat description of the set of rank--one tensors, under the linear isomorphism $\phi$.

First, we note that, under this identification,  the Frobenius scalar product is given by the restriction of the real part of the \emph{Bombieri--Weyl} Hermitian product, defined on complex polynomials by
\begin{equation}\label{eq:BWdef}\langle p_1, p_2\rangle_{\mathrm{BW}}:=\frac{1}{\pi^{n+1}}\int_{\C^{n+1}}p_1(z)\overline{p_2(z)}e^{-\|z\|^2}\mathrm{d} z,\end{equation}
where $\mathrm{d}z:=(i/2)^{n}\mathrm{d}z_0\mathrm{d}\overline{z_0}\ldots \mathrm{d}z_n\mathrm{d}\overline{z_n}$ is the Lebesgue measure (we show that the isomorphism \eqref{eq:phi} is an isometry in Proposition \ref{propo:phi}). This defines the unique, up to multiples, Hermitian product on the space of complex polynomials which is invariant under the action of the unitary group by change of variables. The restriction of the real part of this Hermitian product to the space of real polynomials is still called the \emph{Bombieri--Weyl} scalar product; the above unitary invariance implies its invariance under the action of the orthogonal group by change of variables. 

The space of homogeneous polynomials with the Gaussian measure obtained by pushworfard under the linear isomorphism $\phi$  of the Gaussian measure \eqref{eq:gaussiantensor} is sometimes called the \emph{Kostlan Ensemble} or the \emph{Shub--Smale Ensemble}, see Remark \ref{remarkBW}.

Finally, under the map $\phi$ the set of rank--one tensors can be identified with the \emph{Veronese variety} $\W_{n,d}\subset \R[x_0, \ldots, x_n]_{(d)}$ of signed $d$--th powers of linear forms. In fact, given a rank--one tensor $T=\pm v^{\otimes d}$ the corresponding homogeneous polynomial is 
$$p_T(x)=\langle \pm v^{\otimes d}, x^{\otimes d}\rangle_{F}=\pm \langle v,x\rangle^d,$$
which is the (signed) $d$--th power of the linear form $x\mapsto \langle v,x\rangle$.

At this point we are in the position of giving a more geometric formulation of our question above, which therefore requires computing, for $\delta>0$ small enough, the quantity:
$$\mathbb{P}\bigg\{p\in \R[x_0, \ldots, x_n]_{(d)}\,\bigg|\, \mathrm{dist}_{\mathrm{BW}}(p, \W_{n,d})\leq \delta \|p\|_{\mathrm{BW}}\bigg\}.$$ In this way, we can regard the above probability as the normalized volume of a tubular neighbourhood of the intersection of the Veronese variety with the unit sphere in the Bombieri--Weyl norm. Thus, our question becomes:
 \[\textit{\lq\lq What is the volume of a neighbourhood of the spherical Veronese variety?\rq\rq}\]
In this paper, exploiting Weyl's Tube Formula, we derive an exact expression for the above volume, for small enough neighbourhoods. Moreover, as a byproduct of our computations, we give a lower bound on the size of the neighbourhood of the set of rank--one tensors that admit a unique best rank--one approximation.

 
 \begin{remark}\label{remarkBW}The properties of the Bombieri--Weyl distribution on the space of real (and complex) polynomials have been intensively studied, starting from the influential works of A. Edelman, E. Kostlan, M. Shub and S. Smale \cite{EdelmanKostlan95, CB1, CB2, CB3}. The point of view of random tensors has been adopted first by E. Horobet and J. Draisma in \cite{DHrankone} and by P. Breiding in \cite{Breidingeigenvalues} for the study of the expected number of eigenvalues of a random symmetric tensor, with respect to the Bombieri--Weyl distribution. Under the identification between symmetric tensors and homogeneous polynomials, eigenvalues correspond to critical values of the restriction of the polynomial to the unit sphere. Eigenvectors correspond to critical points of the polynomial: under the Veronese embedding these critical points give rank--one tensors that are critical points of the distance function on the Veronese variety from the given tensor.  Among these critical points (which are rank--one tensors) the closest to the original tensor are its best rank--one approximations. In \cite{Breidingeigenvalues} the average number of such critical points is computed. In this paper we will instead give the size and estimate the probability of the set of tensors which admit a unique best rank--one approximation. 
\end{remark} 

The use of Weyl's Tube Formula is fairly standard for results of this type \cite{Burgissercondition, BL}: it allows to deduce an \emph{exact} expression, for $\varepsilon>0$ small enough, of the volume of an $\varepsilon$--neighbourhood of a smooth submanifold $W$ of the sphere, or the euclidean space, as a function of some differential--geometric quantities of $W$, called its curvature coefficients. Our main contribution is the nontrivial computation of the curvature coefficients of the spherical Veronese variety and the explicit quantification of the above expression ``for $\varepsilon>0$ small enough'' for this variety, through the computation of its reach.\par 
 One could generalize this question to higher ranks by looking at secant varieties to the Veronese, whose geometry has been intensively studied, see \cite{secantlectures} for a survey. We propose to investigate this in future works.
 
 We now describe the main ingredients and state the main results of our work in more detail.

\subsection{The spherical Veronese}
The main object we consider in this work is the real spherical Veronese variety $\SW_{n,d}$, which is the intersection of the Veronese variety in $\R[x_0, \ldots, x_n]_{(d)}\simeq \R^{N+1}$ with the unit sphere for the Bombieri--Weyl norm:
$$\SW_{n,d}:=\W_{n,d}\cap S^N.$$ 
Here $N+1={n+d\choose d}$ is the dimension of the space of symmetric tensors, or equivalently, of homogeneous polynomials.
We regard the set $\SW_{n,d}$ as the image of the spherical Veronese embedding associated to the Bombieri--Weyl basis, or its double copy, depending on the parity of $d$. This embedding is the smooth map $\widehat{\nu}_{n,d}:S^n\to S^N$ given by
\[x \xmapsto{\widehat \nu_{n,d}} \biggl(\binom{d}{\alpha}^{\frac{1}{2}}x^{\alpha}\biggr)_{\alpha},\]
where $\alpha \in \Z^{n+1}_{\geq0}$ satisfy $\alpha_0+\dots+\alpha_n=d$, $\binom{d}{\alpha}$ is the multinomial coefficient, and $S^n$ is the euclidean sphere in $\R^{n+1}$. Denoting $\widehat{\nu}_{n,d}(S^n)$ by $\SV_{n,d}$, we see that
$$\SW_{n,d}=\SV_{n,d}\cup -\SV_{n,d},$$
where $\SV_{n,d}=-\SV_{n,d}$ if $d$ is odd and $\SV_{n,d}\cap (-\SV_{n,d})=\emptyset$ if $d$ is even. For this reason, we will call $\SV_{n,d}$ the \emph{spherical Veronese surface}, to distinguish it from the spherical Veronese variety $\SW_{n,d}$, in the case $d$ is even. In the projective picture, the difference between the two ceases to exist:
$$\PV_{n,d}:=\mathrm{P}(\SV_{n,d})=\mathrm{P}(\SW_{n,d})\subset \R\mathrm{P}^{N},$$
where $ \R\mathrm{P}^{N}$ denotes the projectivization of the space of real homogeneous polynomials.
Recall that $\SV_{n,d}$ parametrizes the $d$--th powers of norm--$1$ linear  forms on $\R^{n+1}$ and, therefore, rank--one and norm--one tensors up to signs. Hence, the spherical Veronese surface $\SV_{n,d}$ corresponds to an orbit for the action of $O(n+1)$ on homogenous polynomials by change of variables. Even more is true: when turning $\SV_{n,d}$ into a Riemannian manifold with the metric induced by the Bombieri--Weyl scalar product, the transitive action of $O(n+1)$ on ${\SV}_{n,d}$ is through isometries induced by isometries of $S^N$, given the invariance property of the Bombieri--Weyl structure. The immediate, yet crucial, consequence is that the extrinsic geometry of the isometric embedding $\SV_{n,d} \hookrightarrow S^N$ is exactly the same at every point. The same conclusion clearly holds for $\SW_{n,d}\hookrightarrow S^N$. \par

\subsection{Weyl's tube formula and the reach of an embedding}
Let $(\overline{M},g)$ be a Riemannian manifold and $M \hookrightarrow \overline{M}$ be an isometric embedding of a compact smooth submanifold. We can consider the set of points in $\overline{M}$ at distance less than a given $\varepsilon>0$ from $M$ and call such a set a \textit{tubular neighbourhood} of $M$ in $\overline{M}$ of radius $\varepsilon$, denoted as $\mathcal{U}(M,\varepsilon)$. 

It is well known that for smooth compact embeddings $M \hookrightarrow \overline{M}$ and small enough radii, the exponential map on the normal bundle provides a smooth parametrization of the tubular neighbourhood. 
This description is what really underlies the celebrated \lq\lq Weyl's tube formula\rq\rq\ (\cite{Weyltubes}), which constitutes one of the main tools to compute the volume of tubular neighbourhoods in a euclidean or spherical ambient space. This formula expresses the volume as the linear combination
\[\mathrm{Vol}(\mathcal{U}(M,\varepsilon))= \sum_{0\leq e \leq n, \ e \ \text{even}} K_{s+e}(M)J_{N,s+e}(\varepsilon),\]
where $N$ is the dimension of the ambient space, $n$ is the dimension of $M$ and $s:=N-n$ is the codimension of the embedding. The functions $J$'s do not depend on the specific submanifold $M$ and are explicitly known in both the euclidean and spherical cases. The most remarkable aspect of the formula is that the coefficients $K$'s are isometric invariants of the embedding and can be expressed in terms of curvature, from which they are named \textit{curvature coefficients} of the embedding. Remark that nowadays Weyl's tube formula has been re-interpreted in the more general framework of \lq\lq integral geometry\rq\rq, which deals with integrals over a submanifold of polynomials in the entries of the second fundamental form. R. Howard in \cite{Howardkinematic} showed  how the above formula fits in this context and gave a full characterization of the polynomials appearing in Weyl's work. 
\par \smallskip

In the case of the Veronese \emph{variety} ${\SW}_{n,d} \hookrightarrow S^N$, the tubular neighbourhood $\mathcal{U}({\SW}_{n,d},\varepsilon)$ gives a description of the norm--$1$ symmetric tensors that are $\varepsilon$--close to a rank--$1$ tensor in the Bombieri--Weyl metric, in the ambient sphere. As already pointed out, it follows that asking for the probability for a symmetric norm--$1$ tensor to be close to rank--$1$ boils down to computing the normalized volume of this tubular neighbourhood. \par 

For practical reasons, in the paper we will work with $\SV_{n,d}$ instead of its ``double'' $\SW_{n,d}$. There are, however, two technical issues to consider here. The first one is that there will be a factor of $2$ to be taken into account when switching from the Veronese surface $\SV_{n,d}$ to the rank--one variety $\SW_{n,d}$, depending on the parity of $d$. The second one is that the intersection of a $\delta$--neighbourood of the set of rank--one tensors with the unit sphere becomes an $\varepsilon$--neighbourhood of $\SW_{n,d}$ in the unit sphere, with 
$$\varepsilon=\mathrm{arcsin}(\delta).$$ This is why we will use the parameter ``$\varepsilon$'' to formulate the results on the sphere and the parameter ``$\delta$'' for the results in the vector space of tensors.

\subsection{The reach of the Veronese variety}Our aim is to exploit Weyl's tube formula to compute the volume of $\mathcal{U}({\SW}_{n,d},\varepsilon)$. This requires first of all the knowledge of the radii for which the above expression holds. Since the formula is based on the parametrization through the normal exponential map, the supremum of the radii for which this is a good parametrization, or at least a lower bound on that, is what we need to understand to meaningfully use Weyl's result. This quantity is usually called the \textit{reach} of the embedding $M \hookrightarrow \overline{M}$ and in general computing it is a very difficult task, often unfeasible since it requires to study not only how normal geodesics originating from every point of the submanifold behave, but also how and when geodesics starting from different points cross each other, in order to avoid overlappings in the image. 

In our case, recalling the invariance property of ${\SW}_{n,d} \hookrightarrow S^N$ under the action of the orthogonal group $O(n+1)$, we do not need to study normal geodesics originating from any point, but it is enough to choose a specific one and perform computations involving only geodesics originating from this chosen one. This drastically reduces the complexity of the computation, allowing us to obtain the following result, stated in a more detailed form in Theorem \ref{reachtheorem}.\begin{thm}[The reach of the spherical Veronese]
The reach of the spherical Veronese variety ${\SW}_{n,d} \hookrightarrow S^N$ is given by
\begin{align*}
    \rho({\SW}_{n,d})=\begin{cases}
    \frac{\pi}{4} & 2\leq d\leq 5\\
        \frac{1}{\sqrt{2}}\sqrt{1+\frac{1}{d-1}} & d\geq 6
    \end{cases}.
\end{align*}
The same result holds for the reach of the Veronese surface $\SV_{n,d}\hookrightarrow S^{N}.$
\end{thm}
Given the interpretation of the neighbourhood of the Veronese variety in terms of symmetric tensors already discussed, this theorem has some interesting consequences. In fact, by \cite{friedland}, the best rank-one approximation of a symmetric tensor is still symmetric and, by construction, every real symmetric tensor which is sufficiently close to rank--one tensors admits a unique best rank--one approximation (see Corollary \ref{bestrankonecorollary}), therefore we can restate the result as follows.

\begin{cor}[Uniqueness of best rank--one approximations]\label{bestrankonecorollaryintro}
Every symmetric tensor $p$ at distance less than $\sin(\rho(V_{n,d})) \|p\|_{\mathrm{BW}}$ from rank--one admits a unique best rank--one approximation.\end{cor}
In particular, the normalized volume of a neighbourhood of the Veronese variety of radius $\rho(V_{n,d})$ would therefore provide a lower bound for the probability that such tensors have a unique best rank--one approximation. This bound can be explicitly computed by plugging in $\varepsilon=\rho(V_{n,d})$ in Theorem \ref{finalvolumetheorem}. 

\begin{remark}[Bottlenecks width of the Veronese variety]From Theorem \ref{reachtheorem} it actually follows that the width of the narrowest bottleneck of the Veronese variety is $\frac{\pi}{2}$. Bottlenecks of algebraic varieties play an important role in optimization problems. For a general algebraic variety, the number of bottlenecks can be estimated using classical invariants such as Chern classes and polar classes \cite{dirocco1}. For a general discussion on this subject see \cite{dirocco1, dirocco2}. 
\end{remark}

\subsection{The curvature coefficients of the Veronese variety}
The other ingredient needed in Weyl's formula is the curvature properties of the embedding, in particular the Weingarten operator along normal directions, which encodes the second fundamental form. Again by the invariance of the extrinsic geometry of ${\SW}_{n,d} \hookrightarrow S^N$, it is enough to compute this at a specific point, which we choose to be $x_0^d$ for simplicity of computations. \par \smallskip
Before stating our result, recall that we have denoted by $\mathrm{GOE}(n)$ the \emph{Gaussian Orthogonal Ensemble}, i.e. the set $\mathrm{Sym}(n, \R)$ endowed with the Gaussian probability distribution coming from the Frobenius scalar product, see Section \ref{GOEKostlansection} for more details. \par \smallskip
Our main result on the extrinsic geometry of the embedding $\SW_{n,d}\hookrightarrow S^N$ is the following, and we refer to Theorem \ref{normaldecompositiontheorem} for a more comprehensive statement. 
\begin{thm}[The normal bundle splitting]\label{thm:decointro}Let $p\in {\SW}_{n,d}$ and denote by $L_{\eta}$ the symmetric matrix representing second fundamental form of ${\SW}_{n,d} \hookrightarrow S^N$ along a normal direction $\eta \in N_{p}{\SW}_{n,d}$. There exists an orthogonal decomposition $N_{p}{\SW}_{n,d} = W \oplus P$ such that the following statements hold: 
\begin{enumerate}
    \item $L_{\eta}=0$ for every $\eta \in P$;
    \item $W$ with its induced Bombieri--Weyl metric is isometric to $\mathrm{Sym}(n)$ with the Frobenius one. Moreover, if we pick $\eta \in W$ Gaussian, then $L_{\eta} \sim \sqrt{2}\biggl(\frac{d-1}{d}\biggr)^{\frac{1}{2}}\mathrm{GOE}(n)$.
\end{enumerate}  
\end{thm}
This theorem could find applications going beyond the scope of this paper. It gives a full description of the second fundamental form in terms of $\mathrm{GOE}(n)$ matrices. Using this description in Weyl's tube formula to compute the curvature coefficients of the spherical Veronese, the consequence is that the computation of some integrals on the normal bundle depending on the second fundamental form boils down to computing the expectation of a determinant involving $\mathrm{GOE}(n)$ matrices. This is reduced to an easy, purely combinatorial computation (see Appendix \ref{expectationappendix}) and thus we obtain the following explicit expressions for the curvature coefficients.
\begin{thm}[The curvature coefficients of the spherical Veronese]
The curvature coefficients of the Veronese variety ${\SW}_{n,d} \hookrightarrow S^N$ are as follows:
\begin{align*}
    K_{N-n+j}({\SW}_{n,d})=(-1)^{\frac{j}{2}}d^{\frac{n}{2}}\biggl(\frac{d-1}{d}\biggr)^{\frac{j}{2}}
    \frac{2^{n+2-j}\pi^{\frac{N}{2}}\Gamma\left(\frac{n}{2}+1\right)}{\Gamma\left(\frac{j}{2}+1\right)\Gamma(n+1-j)\Gamma\left(\frac{N+j-n}{2}\right)}
    \end{align*}

for $0 \leq j \leq n$ and $j$ even, and $K_{N-n+j}({\SW}_{n,d})=0$
otherwise. 
\end{thm}
\noindent We remark that similar results hold true for the projective Veronese variety, using the double covering $S^N \longrightarrow \mathds{RP}^N$, and for the spherical Veronese surface. \par
Plugging these coefficients back in Weyl's tube formula we also obtain the explicit expression of the volume of the tubular neighbourhood for radii smaller than the reach (see Theorem \ref{finalvolumetheorem}), in particular giving an answer to question stated at the beginning of the paper.
\begin{thm}[The probability of being close to rank--one]\label{thm:volumeintro}
Let $p$ be a random Bombieri--Weyl symmetric tensor of order $d$ on $\R^{n+1}$ and $\W_{n,d} \subset \R[x_0,\dots,x_n]_{(d)}\simeq \R^{N+1}$ be the Veronese variety of rank--one tensors. For every $\delta$ such that $0\leq\mathrm{arcsin}(\delta) < \frac{1}{\sqrt{2}}\sqrt{1+\frac{1}{d-1}}$ we have
\begin{align*}
    \mathbb{P}\big(\mathrm{dist}_{\mathrm{BW}}(p,\W_{n,d})\leq \delta\|p\|_{\mathrm{BW}}\big)  = & \sum_{\substack{0\le j \le n \\ \text{ \ j even}}}(-1)^{\frac{j}{2}}d^{\frac{n}{2}}\left(\frac{d-1}{d}\right)^{\frac{j}{2}} 2^{n-j+1}\pi^{-\frac{1}{2}}\\
    &\cdot \frac{\Gamma\left(\frac{n}{2}+1\right) \Gamma\left(\frac{N+1}{2}\right) }{\Gamma\left(\frac{j}{2}+1\right)\Gamma(n+1-j)\Gamma\left(\frac{N+j-n}{2}\right)}\int_0^{\frac{\delta}{\sqrt{1-\delta^2}}}{\frac{t^{N-n+j-1}}{(1+t^2)^{\frac{N+1}{2}}}\mathrm{dt}.}
\end{align*}
\end{thm}
We remark that the above expression, even if unpleasant, gives an exact formula for our  probability.
 In general this probability has an exponential decay in the codimension of $\SW_{n,d} \hookrightarrow S^N$ for any $n,d$. In the case of rational normal curves, there is only one summand in the above theorem, which can be evaluated exactly, giving
 \begin{equation}\label{eq:rational}\mathbb{P}\left(\mathrm{dist}_{\mathrm{BW}}(p,\W_{1,d})\leq \delta\|p\|_{\mathrm{BW}}\right)=\sqrt{d}\,\delta^{d-1}, \quad \textrm{for $\arcsin(\delta)\leq \rho(V_{1, d})$}.\end{equation}
 In fact in this case   $N=d$ and, as one can easily verify,
 $$\int_0^{\frac{\delta}{\sqrt{1-\delta^2}}}{\frac{t^{d-2}}{(1+t^2)^{\frac{d+1}{2}}}}\mathrm{dt}=\frac{\delta^{d-1}}{d-1}.$$
 Once substituted in the above expression this immediately gives \eqref{eq:rational}.

\subsection{Acknowledgements}We wish to thank Paul Breiding and Sarah Eggleston for useful preliminary comments on the first version of this paper and for pointing out a typo in our computation.

\section{Preliminaries}
\subsection{The Gaussian Orthogonal Ensemble and the Bombieri--Weyl distribution}\label{GOEKostlansection}
In this section, we point out the correspondence between the Gaussian Orthogonal Ensemble on the space of symmetric matrices and the Bombieri--Weyl distribution on the space of homogeneous polynomials. \smallskip \par Let $\mathrm{Sym}(n,\R)$ be the space of symmetric $n\times n$ matrices with real entries and denote by $E_{ij}$ the elementary matrix having all entries $0$ except for the $ij$-th one being $1$. Consider a random matrix 
\begin{align*}
    Q=\sum_{i=1}^n \eta_{ii}E_{ii} + \sum_{i<j=1}^n \eta_{ij}\frac{E_{ij}+E_{ji}}{\sqrt{2}},
\end{align*}
where $\eta_{ij}\sim N(0,1)$ are i.i.d. standard Gaussian variables. Then $Q$ has random Gaussian entries distributed as $N(0,1)$ on the diagonal and $N(0,\frac{1}{2})$ off-diagonal, independent except for the obvious symmetry condition. The probability distribution on $\mathrm{Sym}(n,\R)$ induced by such matrices is called the \textit{Gaussian Orthogonal Ensemble} and it is denoted by $\mathrm{GOE}(n)$. This is the standard Gaussian probability distribution associated to the Frobenius scalar product given by $\langle A,B \rangle = tr(AB)$ and therefore for every open set $U\subset \mathrm{Sym}(n,\R)$ 
\begin{align}\label{GOEdensity}
    \mathbb{P}\{Q \in U\} = \frac{1}{(2\pi)^{\frac{n(n+1)}{4}}}\int_{U}e^{-\frac{1}{2}tr(A^2)}dA.
\end{align} 
Recall that the orthogonal group $O(n)$ acts on $\mathrm{Sym}(n,\R)$ by congruence. We denote this action by $\psi:O(n) \longrightarrow GL(\mathrm{Sym}(n,\R))$: for every $R \in O(n)$ and $A \in \mathrm{Sym}(n,\R)$, $\psi(R)(A)=R^tAR$, where $R^t$ denotes the transpose of $R$. \par \smallskip
Let $\C[x_1,\dots,x_n]_{(d)}$ be the space of complex homogeneous polynomials of degree $d$ in $n$ variables. Denote by $\rho:U(n) \longrightarrow GL(\C[x_1,\dots,x_n]_{(d)})$ the action of the unitary group by change of variables, i.e. for every $R \in U(n)$ and $p \in \C[x_1,\dots,x_n]_{(d)}$ set $\rho(R)(p):= p \ \circ R^{-1}$. It is known that $\rho$ is irreducible (see \cite{Itzykson}) and therefore by Schur's lemma and the compactness of $U(n)$ it follows that there exists a unique (up to multiples) Hermitian structure on $\C[x_1,\dots,x_n]_{(d)}$ which is $\rho$-invariant, and which is given by \eqref{eq:BWdef}. Up to scaling, this is the Hermitian structure having the set
\begin{align}\label{BWbasis}
    \bigg\{\binom{d}{\alpha}^{\frac{1}{2}}x^{\alpha}\bigg\}_{\alpha=(\alpha_1,\dots,\alpha_n)\in\Z^n_{\geq0},\ \alpha_1+\dots+\alpha_n=d}
\end{align}
as an orthonormal basis, where $\binom{d}{\alpha} = \frac{d!}{\alpha_1!\dots\alpha_n!}$ and $x^{\alpha}=x_1^{\alpha_1}\dots x_n^{\alpha_n}$. We call this the \textit{Bombieri--Weyl} or \textit{Kostlan} Hermitian structure on $\C[x_1,\dots,x_n]_{(d)}$. \smallskip \par Restricting to real homogeneous polynomials $\R[x_1,\dots,x_n]_{(d)}$,  we define an inner product, which we call again \textit{Bombieri--Weyl}. Notice that \eqref{BWbasis} is also a real orthonormal basis since the polynomials in \eqref{BWbasis} have real coefficients. We also restrict the action $\rho$ to an action of the orthogonal group $O(n)$ on $\R[x_1,\dots,x_n]_{(d)}$ and the inner product we introduced is invariant for this restricted action. Remark that this restricted action is not irreducible anymore (it is a computation to show that the subspace of harmonic polynomials in $\R[x_1,\dots,x_n]_{(d)}$ is a non-trivial invariant subspace) and the Bombieri--Weyl inner product is not the unique invariant one. The standard Gaussian probability distribution on $\R[x_1,\dots,x_n]_{(d)}$ associated to the Bombieri--Weyl inner product is the one induced by the random polynomial \begin{align*}
    P(x)=\ \sum_{\substack{\alpha \in \Z^{n}_{\geq0} \\ \alpha_1+\dots+\alpha_n=d}} \xi_{\alpha}\binom{d}{\alpha}^{\frac{1}{2}}x^{\alpha},
\end{align*} 
where $\xi_{\alpha}$ are i.i.d. standard Gaussians $N(0,1)$. We call it again the \textit{Bombieri--Weyl distribution}.\par \smallskip

The map 
\begin{align}\label{isometryFrobKost}
    \phi:\mathrm{Sym}(n,\R) &\longrightarrow \R[x_1,\dots,x_n]_{(2)} \\ \notag
    Q \ & \longmapsto p_Q=x^tQx
\end{align}
defines an isomorphism of $\mathrm{Sym}(n,\R)$ with $\R[x_1,\dots,x_n]_{(2)}$ and one can check that the orthonormal basis for the Frobenius product given by $\{E_{ii},\frac{E_{ij}+E_{ji}}{\sqrt{2}}\}_{i<j=1,\dots,n}$ is mapped to the orthonormal basis for the Bombieri--Weyl product given by \eqref{BWbasis}. It follows that $\phi$ defines an isometry of $\mathrm{Sym}(n,\R)$ endowed with the Frobenius inner product with $\R[x_1,\dots,x_n]_{(2)}$ endowed with the Bombieri--Weyl one and thus we can identify the corresponding standard Gaussian probability distributions. Even more is true: $\phi$ is an isomorphism between the representations $(\psi,\mathrm{Sym}(n,\R))$ and $(\rho,\R[x_1,\dots,x_n]_{(2)})$, i.e. the following diagram
\begin{align*}\label{isoFrobeniusKostlan}
    \begin{tikzcd}[ampersand replacement=\&]
{\mathrm{Sym}(n,\R)} \arrow[d, "\phi"'] \arrow[r, "\psi(R)"] \& {\mathrm{Sym}(n,\R)} \arrow[d, "\phi"] \\
{\R[x_1,\dots,x_n]_{(2)}} \arrow[r, "\rho(R)"']           \& {\R[x_1,\dots,x_n]_{(2)}}         
\end{tikzcd}
\end{align*}
commutes for every $R \in O(n)$, as one can easily check by a straightforward computation. 

More generally we have the following.
\begin{proposition}\label{propo:phi}The map $\phi$ from \eqref{eq:phi} gives an isometry of Euclidean spaces
$$\phi:\left(\mathcal{S}^d(n), \langle\cdot,\cdot\rangle_{\mathrm{F}}\right)\stackrel{\simeq}{\longrightarrow}\left(\mathbb{R}[x_1, \ldots, x_n]_{(d)}, \langle\cdot,\cdot\rangle_{\mathrm{BW}}\right).$$
\end{proposition}
\begin{proof}
    We have to check that $\phi$ maps an orthonormal basis for $\left(\mathcal{S}^d(n), \langle\cdot,\cdot\rangle_{\mathrm{F}}\right)$ to an orthonormal basis for $\left(\mathbb{R}[x_1, \ldots, x_n]_{(d)}, \langle\cdot,\cdot\rangle_{\mathrm{BW}}\right)$. For $\underline{i}=(i_1,\dots,i_d)$ with $1\leq i_1,\dots,i_d \leq n$, we say that $\underline{i}$ has type $\alpha=(\alpha_1,\dots,\alpha_n)$ if $\alpha_j=\#\{i_k = j\}$. It is straightforward to show that the following is an orthonormal basis for the Frobenius scalar product on $\mathcal{S}^d(n)$:
    \begin{align*}
        \left\{\left(\frac{1}{\alpha_1!\dots\alpha_n!\ d!}\right)^{\frac{1}{2}}\sum_{\sigma \in S_d}e_{i_{\sigma(1)}}\otimes \dots \otimes e_{i_{\sigma(d)}} \ \bigg| \ (i_1,\dots,i_d) \ \text{type $\alpha=(\alpha_1,\dots,\alpha_n)$}, \ 1\leq i_1,\dots,i_d \leq n\right\},
    \end{align*}
    where $S_d$ is the symmetric group on $\{1,\dots,d\}$. Remark that different multi-indices $(i_1,\dots,i_d)$ and $(j_1,\dots,j_d)$ will give the same tensor in the above basis if and only if they have the same type. Therefore we can index the tensors of the basis according to multi-indices $\alpha=(\alpha_1,\dots,\alpha_n)\in \Z^n_{\geq 0}$ with $\alpha_1+\dots+\alpha_n=d$ and we set the notation $v_{\alpha}$ for the tensors above. We claim that $\phi$ sends $v_{\alpha}$ to the polynomial in \eqref{BWbasis} corresponding to multi-index $\alpha$. This is a straightforward computation. Indeed $\phi(v_{\alpha})=p_{v_{\alpha}}$ where
    \begin{align*}
        p_{v_{\alpha}}(x)=&\langle v_{\alpha},x^{\otimes d} \rangle_F \\
        =&\langle \left(\frac{1}{\alpha_1!\dots\alpha_n! \ d!}\right)^{\frac{1}{2}}\sum_{\sigma \in S_d} e_{i_{\sigma_1}}\otimes\dots\otimes e_{i_{\sigma(d)}}, \left(\sum_{j_1=1}^nx_{j_1}e_{j_1}\right)\otimes\dots\otimes\left(\sum_{j_d=1}^n x_{j_d}e_{j_d}\right)\rangle_F \\
        =&\left(\frac{1}{\alpha_1!\dots\alpha_n! \ d!}\right)^{\frac{1}{2}}\sum_{\sigma \in S_d}\sum_{j_1,\dots,j_d=1}^n \langle e_{i_{\sigma_1}}\otimes\dots\otimes e_{i_{\sigma(d)}},x_{j_1}\dots x_{j_d}e_{j_1}\otimes\dots\otimes e_{j_d}\rangle_F \\
        =&\left(\frac{1}{\alpha_1!\dots\alpha_n! \ d!}\right)^{\frac{1}{2}}\sum_{\sigma \in S_d}x_{i_{\sigma(1)}}\dots x_{i_{\sigma(d)}} = \left(\frac{1}{\alpha_1!\dots\alpha_n! \ d!}\right)^{\frac{1}{2}} d! \ x_1^{\alpha_1}\dots x_n^{\alpha_n} = \\
        =& \left(\frac{d!}{\alpha_1!\dots\alpha_n!}\right)^{\frac{1}{2}}x_1^{\alpha_1}\dots x_n^{\alpha_n},
    \end{align*}
    where the second to last equation follows from $x_{i_{\sigma(1)}}\dots x_{i_{\sigma(d)}}=x_1^{\alpha_1}\dots x_n^{\alpha_n}$ for every $\sigma \in S_d$, since permutations of the indices do not change the type of the multi-index, which is always $\alpha=(\alpha_1,\dots,\alpha_n)$. The proof is concluded. 
\end{proof}

\subsection{Tubular neighbourhoods and Weyl's tube formula}\label{tubularsection}
Let $M$ be an isometrically embedded $n$-dimensional submanifold of a Riemannian manifold $(\overline{M},g)$, i.e. $M$ is itself a Riemannian manifold with the metric induced by $(\overline{M},g)$. Denote by $T_pM$, $N_pM$ the tangent and normal spaces to $M$ at $p \in M$. Denote also by $\overline{\nabla}$, $\nabla$ the Riemannian connections of $\overline{M}$ and $M$ respectively. For smooth vector fields $X,Y$ on $M$, consider $\overline{X}$,$\overline{Y}$ their local extensions to smooth vector fields on $\overline{M}$. Then we have 
\[\overline{\nabla}_{\overline{X}}\overline{Y}=\nabla_XY + B(X,Y),\]
where $\nabla_XY$ is the tangential component to $M$ and $B(X,Y)$ is the normal one. By the properties of the Riemannian connection, at every $p \in M$ we can regard $B$ as a symmetric bilinear map $B:T_pM \times T_pM \longrightarrow N_pM$ and we call this the \textit{second fundamental form} of $M \hookrightarrow \overline{M}$ at $p \in M$. Given a normal direction $\eta \in N_pM$ we define the \textit{second fundamental form along $\eta$}, denoted by $H_{\eta}:T_pM \times T_pM \longrightarrow \R$, by projecting $B$ along $\eta$, i.e. for every $v,w \in T_pM$
\[H_{\eta}(v,w)=g(B(v,w),\eta).\]
To this bilinear form we can associate a selfadjoint operator, called the \textit{Weingarten operator along $\eta$} and denoted by $L_{\eta}:T_pM \longrightarrow T_pM$, such that for every $v,w \in T_pM$
\[g(L_{\eta}(v),w)=H_{\eta}(v,w)=g(B(v,w),\eta).\]
This operator will play a key role in Weyl's tube formula.
\begin{remark}\label{Weingartenmatrixremark}
Given $p \in M$, consider a basis $\{e_1,\dots,e_n\}$ of $T_pM$. Set $H_{\eta,ij}=H_{\eta}(e_i,e_j)$ and denote by $L_{\eta}=(L_{\eta,ij})$ the matrix representing the Weingarten operator with respect to the given basis and $g$. It is clear from the definitions that 
\[H_{\eta,ij}=H_{\eta}(e_i,e_j)=g(L_{\eta}(e_i),e_j)=L_{\eta,ij}.\]
Therefore, computing the matrix representing the Weingarten operator with respect to a given basis is equivalent to the computation of the second fundamental form on the elements of that basis.
\end{remark}
\begin{remark}\label{secondfundformremark}
Let $N$ be a $n$-dimensional submanifold of a euclidean space with the induced metric and $\varphi:\R^n \longrightarrow N$ be a parametrization of $N$ around $p\in N$ with coordinates $a_1,\dots,a_n$. A basis for $T_pN$ is given by the vectors 
\[\frac{\partial \varphi}{\partial a_i} := d_{\varphi^{-1}(p)}\varphi\bigg(\frac{\partial}{\partial a_i}\bigg).\]
Since the Christoffel symbols of the Riemannian connection of the euclidean space are all null, it follows that to compute the second fundamental form of $N$ at $p$ along a normal direction $\eta$ it is enough to compute the second derivatives of the parametrization
\[H_{\eta}\biggl( \frac{\partial\varphi}{\partial a_i}, \frac{\partial\varphi}{\partial a_j}\biggr)=\mathlarger{\mathlarger{\langle}} \mathlarger{\nabla}_{ \frac{\partial\varphi}{\partial a_i}} \frac{\partial\varphi}{\partial a_j},\eta \mathlarger{\mathlarger{\rangle}}=\mathlarger{\mathlarger{\langle}} \frac{\partial^2 \varphi}{\partial a_i \partial a_j},\eta \mathlarger{\mathlarger{\rangle}},\]
where $\langle\  , \rangle$ is the euclidean inner product. Let now $N$ be an isometrically embedded submanifold of the sphere $S^l$, where $S^l$ is given the round metric inherited by $\R^l$, and let $\eta$ be a normal vector to $N$ at $p\in N$. Remark that we can also consider $N$ as an isometrically embedded submanifold of $\R^l$. Then the second fundamental form of $N$ along $\eta$ is the same whether we consider $N$ as a submanifold of $S^l$ or of $\R^l$. This means that above formula holds also for submanifolds of round spheres.
\end{remark}
Given $M$, $\overline{M}$ as above, we can define the \textit{$\varepsilon$-small normal bundle} $N^{\varepsilon}M$ as the subset of the normal bundle $NM$ consisting of vectors with norm less than $\varepsilon > 0$. We also call \textit{normal exponential map} the restriction of the exponential map of $\overline{M}$ to $NM$.
\begin{definition}
If there exists an $\varepsilon>0$ such that
\[\exp|_{N^{\varepsilon}M} : N^{\varepsilon}M \longrightarrow \overline{M}\]
is a diffeomorphism on its image, where $\exp$ denotes the exponential map of $\overline{M}$, we call $N^{\varepsilon}M$ a \textit{tubular neighbourhood of $M$ in $\overline{M}$}.
\end{definition}
\noindent Recall the definition of distance of a point $x \in \overline{M}$ from the submanifold $M$, given by $d_g(x,M) := \inf\{d_g(x,y) \ | \ y \in M\}$ where $d_g(x,y)$ is the Riemannian distance between $x,y$. We introduce the following set 
\begin{equation}\label{tubularneighbourhooddef}
    \mathcal{U}(M,\varepsilon) := \{x \in \overline{M} \ | \ d_g(x,M)<\varepsilon\},
\end{equation}
consisting of points at distance less than $\varepsilon>0$ from $M$. The description of this set for submanifolds of $\R^3$ is quite easy: the distance of a point from a surface or a curve will  always be given by the length of a segment starting from the point and meeting the submanifold orthogonally, given that segments are geodesics. This situation can be generalized, as the following theorem shows. Even though this result is well known, we were not able to find a full explicit reference for this general setting. We, therefore, provide a full proof in Appendix \ref{tubularneighbourhoodproof}, filling in the details of the outline given in \cite[Theorem 6.6]{Cannassymplectic}.
\begin{theorem}[\textbf{Tubular neighbourhood theorem}]\label{tubularneighbourhoodtheorem}
Let $M$ be a compact isometrically embedded submanifold of a Riemannian manifold $(\overline{M},g)$. Then there exists an $\varepsilon>0$ small enough such that $\exp|_{N^{\varepsilon}M}:N^{\varepsilon}M \longrightarrow \overline{M}$ is a diffeomorphism on its image and $\exp(N^{\varepsilon}M)=\mathcal{U}(M,\varepsilon)$. 
\end{theorem}
\noindent Theorem \ref{tubularneighbourhoodtheorem} can be seen as an existence result for tubular neighbourhoods of compact submanifolds and as a characterization of the set $\mathcal{U}(M,\varepsilon)$ for $\varepsilon$ small enough. For this reason, in the following, we will refer also to $\mathcal{U}(M,\varepsilon)$ as a tubular neighbourhood.
\begin{remark}
The compactness assumption in theorem \eqref{tubularneighbourhoodtheorem} is crucial. If we remove compactness, we can only prove the existence of a smooth function $\varepsilon(\cdot):M \longrightarrow \R_{>0}$ such that the restriction of the exponential map to $N^{\varepsilon(\cdot)}M$ is an embedding, where $N^{\varepsilon(\cdot)}M=\{v \in N_xM \ | \ x \in M, \ \norm{v}<\varepsilon(x)\}$, i.e. the $\varepsilon$ is not uniform anymore but it depends on the point.
\end{remark}
\begin{definition}
Let $M$ be an isometrically embedded submanifold of a Riemannian manifold $(\overline{M},g)$. We define the \textit{reach of $M \hookrightarrow \overline{M}$} as 
\[\rho(M)=\sup\{\varepsilon\geq0 \ | \ N^{\varepsilon}M \ \text{is a tubular neighbourhood of M}\}.\]
\end{definition}
\noindent We can restate theorem \eqref{tubularneighbourhoodtheorem} by saying that the reach of a compact submanifold is always positive. Remark that even if for brevity we write $\rho(M)$, the reach is not an intrinsic property of $M$ but it depends on the way $M$ is embedded into $\overline{M}$. From the very definition it follows that $\rho(M)$ can be expressed as the minimum between $\rho_1(M)=\sup \{\varepsilon\geq0 \ | \ \exp|_{N^{\varepsilon}M} \ \text{is an immersion} \}$ and $\rho_2(M)=\sup \{\varepsilon\geq0 \ | \ \exp|_{N^{\varepsilon}M} \ \text{is injective}\}$. \\
The points where the differential of the normal exponential map is not injective are called \textit{focal points}. Their existence is linked to the presence of particular Jacobi fields, see \cite[Section 10.4]{doCarmo}. For compact submanifolds of round spheres, one obtains the following expression
\begin{align}\label{rho_1}
    (\rho_1(M))^{-1}=\underset{x \in M}{\sup} \ \sup\{\norm{\overset{..}\gamma(0)} \ s.t. \ \gamma:(-\delta,\delta) \longrightarrow M \\ \notag\textrm{arclength geodesic}  \textrm{ in $M$ with}\ \gamma(0)=x  \}.
\end{align}
If $\varepsilon>0$ is the first value for which we lose injectivity of the restriction of the normal exponential, it means that there is a point in $\overline{M}$ that is reached by two length--$\varepsilon$ normal geodesics starting from two different points of $M$. Using the Generalized Gauss Lemma for geodesic variations, see \cite[Lemma 2.11]{Graytubes}, and uniqueness of geodesics, one can show that the union of these two normal geodesics is again a geodesic meeting $M$ orthogonally at its endpoints. Therefore we have the following expression for $\rho_2(M)$
\begin{align}\label{rho_2}
    \rho_2(M)=\frac{1}{2}\ \inf \{l(\gamma) \ | \ \gamma:[a,b]\longrightarrow \overline{M} \ \textrm{geodesic s.t.}\  \gamma(a),\gamma(b) \in M,\\ \notag \dot{\gamma}(a) \in N_{\gamma(a)}M, \dot{\gamma}(b) \in N_{\gamma(b)}M \}.
\end{align}
We will apply these expressions in chapter \ref{reachchapter} to compute the reach of the Veronese variety. \par \smallskip
In \cite{Weyltubes} Weyl presented a fundamental work, answering a question posed by Harold Hotelling: how can we compute the volume of a \lq\lq tube\rq\rq, i.e. a tubular neighbourhood, of fixed radius around a closed $n$--dimensional manifold in $\mathds{R}^N$ or $S^N$? Hotelling himself answered the case of curves, both in the euclidean and in the spherical setting. Weyl extended these results to any dimension $n$ as follows.
\begin{theorem}[\textbf{Weyl's tube formula}]\label{Weyltheorem}
Let $M$ be a smooth, $n$--dimensional, compact submanifold isometrically embedded in $\R^N$ (or $S^N$) with their standard metrics. Then, for $\varepsilon<\rho(M)$, the following formula holds:
\begin{align}\label{implicitWeyl}
    \mathrm{Vol}\bigl(\mathcal{U}(M,\varepsilon)\bigr)=\sum_{0\leq e \leq n,  \ e \ \text{even}}K_{s+e}(M)J_{N,s+e}(\varepsilon),
\end{align}
where $s:=N-n$ is the codimension of $M$ and $J_{N,s+e}$ are linearly independent functions of $\varepsilon$ only. In the euclidean case, the universal functions $J$ have the following form:
\begin{align}\label{Jeuclidean}
    J_{N,k}(\varepsilon):=\varepsilon^{k},
\end{align}
while in the spherical one, they are given by
\begin{align}\label{Jspherical}
    J_{N,k}(\varepsilon):=\int_{0}^{\varepsilon}(\sin\rho)^{k-1}(\cos\rho)^{N-k}\, d\rho = \int_0^{\tan\varepsilon} \frac{t^{k-1}}{(1+t^2)^{\frac{N+1}{2}}}\, dt.
\end{align}
Moreover, the coefficients $K_j(M)$ are isometric invariants of $M$.
\end{theorem}
\noindent The coefficients $K_{j}(M)$ are called the \textit{curvature coefficients} of $M$. The motivation for this comes from the fact that they are integrals of functions on the second fundamental form of $M$.
Notice that in \cite{Weyltubes} Weyl uses different normalization constants for \eqref{Jeuclidean} and \eqref{Jspherical}. We chose to follow the normalization introduced by Nijenhuis in \cite{Nijenhuischern}, which proves to be handier for applications, see for instance \cite{Burgisseraverage}. Remark that the $\varepsilon$'s for which the formula holds depend on the reach and therefore on the embedding, while  isometric embeddings will give the same curvature coefficients. The dependence of the validity of the formula on the reach is due to the proof relying on parametrizing the tubular neighbourhood through the normal exponential map. In \cite{Howardkinematic} Howard contextualizes Weyl's result in the framework of \lq\lq Integral Geometry\rq\rq, where he considers more general integrals of polynomials on the second fundamental forms of submanifolds of homogeneous spaces.\\
There are explicit integral versions of formula \eqref{implicitWeyl} for both the euclidean and spherical cases. For the latter, with the same notations above, this reads as 
\begin{align}\label{sphericalWeyl}
    \mathrm{Vol}\bigl(\mathcal{U}(M,\varepsilon)\bigr)=\int_{p\in M}\int_{t=0}^{\tan \varepsilon}\int_{S(N_pM)}\frac{t^{m-1}\det\big[I_n-tL_{\eta}\big]}{(1+t^2)^{\frac{N+1}{2}}}\,\mathrm{vol}_M\,d\eta\,dt,
\end{align}
where $S(N_pM)$ denotes the unit sphere in $N_pM$, $I_n$ is the $n \times n$ identity matrix, $L_{\eta}$ is the Weingarten operator of $M\hookrightarrow \overline{M}$ along the unit normal vector $\eta$, and $d\eta$ is a short notation for the volume form on $S(N_pM)$. If one explicitly  develops the determinant, it is easy to get back formula \eqref{implicitWeyl}.

\subsection{A lemma on integration on spheres}\label{sectionlemma}
Consider a sphere $S^m$ for some $m \geq 2$ and fix $k \in \{1,\dots,m-1\}$. Denote by $\iota:S^k \hookrightarrow \R^{k+1}$ the inclusion map and consider the map
\begin{align}\label{parametrizationlemma}
    S^k \times \overset{\circ}{D}{} ^{m-k} & \overset{\varphi}{\longrightarrow} \ S^m \subset \R^{m+1} = \R^{k+1}\times \R^{m-k} \\ \notag
    (\sigma, z) \qquad & \longmapsto \qquad (\sqrt{1-\abs{z}^2}\ \iota(\sigma),\ z),
\end{align} \smallskip
\noindent giving a smooth parametrization of $S^m\setminus \big\{\{0\} \times S^{m-k-1} \big\}\subset \R^{k+1}\times \R^{m-k}$. For every $l \in \N$, consider $\R^{l}$ endowed with a non--degenerate scalar product and coordinate functions $x_1,\dots,x_l$ with respect to an orthonormal basis. Then we have the standard volume form $\mathrm{vol}_{\R^l}=dx_1\wedge\dots\wedge dx_l$ which induces a volume form $\mathrm{vol}_{\overset{\circ}{D}{} ^l}$ on the open norm--$1$ disc $\overset{\circ}{D}{} ^l$ and through the pullback of the inclusion also a volume form $\mathrm{vol}_{S^{l-1}}$ on $S^{l-1}$. Notice that with respect to $\mathrm{vol}_{S^{m}}$, the part of $S^m$ not parametrized by $\varphi$ in \eqref{parametrizationlemma} has measure $0$. We have the following result about integration on spheres, see Appendix \ref{integrationappendix} for a proof.
\begin{lemma}\label{integrationlemma}
With the same notations above, the pullback of $\mathrm{vol}_{S^m}$ through $\varphi$ is given by
\begin{align*}
    \varphi^*(\mathrm{vol}_{S^m}) = \bigl(1-\abs{z}^2\bigr)^{\frac{k-1}{2}}\ \mathrm{vol}_{S^k}\wedge \mathrm{vol}_{\overset{\circ}{D}{} ^{m-k}}. 
\end{align*}
In particular, this implies that
\begin{align}\label{splitintegral}
    \int_{S^m}{f(p) \ \mathrm{vol}_{S^m}} = \int_{S^k}{\int_{\overset{\circ}{D}{} ^{m-k}}{f\bigl(\sqrt{1-\abs{z}^2}\ \iota(\sigma),\ z\bigr)\bigl(1-\abs{z}^2\bigr)^{\frac{k-1}{2}}\ \mathrm{vol}_{S^k}\wedge \mathrm{vol}_{\overset{\circ}{D}{} ^{m-k}}}},
\end{align}
for any measurable function $f$ on $S^m$.
\end{lemma}

\section{The Veronese variety}
Consider the space of homogeneous polynomials of degree $d$ in $n+1$ variables $\R[x_0,\dots,x_n]_{(d)}\cong \R^{N+1}$, where $N:=\binom{n+d}{d}-1$, with the basis described in $\eqref{BWbasis}$. 
\begin{definition}
For $n\geq 1$ and $d \geq 1$, the \textit{real Bombieri--Weyl Veronese embedding} is the map \begin{align*}
     \nu_{n,d}: \mathds{RP}^n & \longrightarrow \ \ \mathds{RP}^{N} \\
    [a] \ \ & \longmapsto \biggl[\binom{d}{\alpha}^{1/2}a^{\alpha}\biggr]
\end{align*}
and it is the Veronese projective embedding associated to the Bombieri--Weyl basis. The \textit{Bombieri--Weyl Veronese variety} is the image of this embedding, denoted by $\PV_{n,d}:=\mathrm{im}(\nu_{n,d})$.
\end{definition}
\noindent The main object we will consider in what follows is the spherical counterpart of $\PV_{n,d}$.
\begin{definition}
The \textit{spherical (Bombieri--Weyl) Veronese map} is the map
\begin{align*}\label{sphericalVeronese}
    \widehat\nu_{n,d}: S^n & \longrightarrow \ S^N \\ \notag
    a \ \ & \longmapsto \biggl(\binom{d}{\alpha}^{1/2}a^{\alpha}\biggr).
\end{align*}
The \textit{spherical (Bombieri--Weyl) Veronese surface} is the image of this map, denoted by $\SV_{n,d}:=\mathrm{im}(\widehat\nu_{n,d})$.
\end{definition}

It is worth stressing in the definition of $\widehat\nu_{n,d}$ that  $S^n$ is the sphere with respect to the standard euclidean product in $\R^{n+1}$ while $S^N$ is the sphere with respect to the Bombieri--Weyl product in $\R[x_0,\dots,x_n]_{(d)}$ and that $\widehat\nu_{n,d}$ is well defined, as one can check by an explicit computation. \par \smallskip
The objects we just introduced have a particularly useful description. Recall that to each $b=(b_0,\dots,b_n) \in \R^{n+1}$ we can associate the linear form on $\R^{n+1}$ given by $l_b(x_0,\dots,x_n)=b_0x_0+\dots+b_nx_n$. It is known that $\PV_{n,d}$ parametrizes projective classes of $d-$th powers of linear forms
\begin{align*}
\PV_{n,d} & =  \big\{[d\text{--th powers of linear forms on} \ \mathds{R}^{n+1}]\big\} \\ & = \big\{\ [a_{\alpha}] \in \mathds{RP}^N \ | \ \exists \ b=(b_0,\dots,b_n) \in \mathds{R}^{n+1} \ \text{s.t.} \ a_{\alpha_0,\dots,\alpha_n}=\binom{d}{\alpha}^{1/2}b_0^{\alpha_0}\dots b_n^{\alpha_n} \big\}
\end{align*}
as one can prove by showing that $\nu_{n,d}([b_0,\dots,b_n])=[(b_0x_0+\dots+b_nx_n)^d]$. This also leads to the well-known description of the Veronese variety $\PV_{n,d}$ as the variety of symmetric decomposable $d-$tensors on $\R^{n+1}$ and is one of the main reasons Veronese varieties have been so intensively studied. A similar description holds for the spherical Veronese surface
\begin{align*}
    \SV_{n,d} & =  \{d\text{--th powers of norm--1 linear forms on} \ \mathds{R}^{n+1}\} \\ & = \{\ (a_{\alpha}) \in S^N \ | \ \exists \ b=(b_0,\dots,b_n) \in S^n \ \text{s.t.} \ a_{\alpha_0,\dots,\alpha_n}=\binom{d}{\alpha}^{1/2}b_0^{\alpha_0}\dots b_n^{\alpha_n} \}.
\end{align*}
Using these descriptions of $\PV_{n,d}$ and $\SV_{n,d}$, it is immediate to prove the following.
\begin{proposition}
$\PV_{n,d}$ is an orbit for the action of the orthogonal group $O(n+1)$ on $\mathds{RP}^N=\mathds{P}(\mathds{R}[x_0,\dots,x_n]_{(d)})$ by change of variables. Similarly $\SV_{n,d}$ is an orbit for the same action of the orthogonal group $O(n+1)$ on $S^N=S(\R[x_0,\dots,x_n]_{(d)})$.
\end{proposition}
Recall the two-fold covering map $\pi_N:S^N \longrightarrow \mathds{RP}^N$ given by the identification of antipodal points. Its restriction to the spherical Veronese $\SW_{n,d}:=\W_{n,d}\cap S^{N}$ gives a covering map $\widehat \pi_{n,d}:\SW_{n,d}\longrightarrow \PV_{n,d}$ whose degree depends on the parity of $d$: if $d$ is odd $\widehat \pi_{n,d}$ is a $2:1$ covering, while if $d$ is even it is $1:1$, since in this case for $b\in S^n$ we have $\widehat\nu_{n,d}(b)=\widehat\nu_{n,d}(-b)$. \par \smallskip
We now turn to metric properties of Veronese manifolds. Consider on $\R^{n+1}$ the standard euclidean metric and on $\R^{N+1}=\R[x_0,\dots,x_n]_{(d)}$ the Bombieri--Weyl one. The metrics induced on $S^n$ and $S^N$ respectively are invariant under the antipodal map and therefore induce metrics on the corresponding projective spaces $\mathds{RP}^n$ and $\mathds{RP}^N$. We denote the metrics on the spheres by $g_{S^n}$ and $g_{S^N}$ and those on the projective spaces by $g_{\mathds{RP}^n}$ and $g_{\mathds{RP}^N}$. Since the covering maps $\pi_n$,\ $\pi_N$ are Riemannian coverings with these metrics, any relation between $g_{S^n}$ and $g_{S^N}$ will also hold between $g_{\mathds{RP}^n}$ and $g_{\mathds{RP}^N}$ and viceversa. Through a direct computation, one can prove the following result.
\begin{proposition}\label{metricrelation}
Pulling back the Bombieri--Weyl metric through the Veronese embedding $\nu_{n,d}: \mathds{RP}^n \longrightarrow \mathds{RP}^N$, for any $n\geq 1$, $d\geq 1$ we have 
\begin{align} 
\nu_{n,d}^* \ g_{\mathds{RP}^N} = \sqrt{d} \ g_{\mathds{RP}^n}.
\end{align}
\end{proposition}
\begin{corollary}\label{submanifoldvolume}
For every $n,d\in \mathbb{N}$ and any smooth $n$--dimensional submanifold $C \hookrightarrow \SV_{n,d}$ we have
\begin{align}\label{explicitvolumesubmanifold}
    \mathrm{Vol}_n^{BW}(C) = \begin{cases}
    \frac{1}{2}d^{\frac{n}{2}} \mathrm{Vol}_n\bigl(\widehat\nu_{n,d}^{-1}(C)\bigr) \ \ \ \text{for $d$ even} \\ \\
    d^{\frac{n}{2}} \mathrm{Vol}_n\bigl(\widehat\nu_{n,d}^{-1}(C)\bigr) \ \ \ \text{for $d$ odd}
    \end{cases},
\end{align}
where $\mathrm{Vol}_n$ is the n--dimensional volume with respect to $g_{S^n}$ and $\mathrm{Vol}_n^{BW}$ is the n--dimensional volume with respect to the metric induced by $g_{S^N}$ on $\SV_{n,d}$. In particular
\begin{align}\label{explicitVeronesevolume}
    \mathrm{Vol}_n^{BW}(\SV_{n,d})= 
    \begin{cases}
    d^{\frac{n}{2}}\frac{\pi^{\frac{n+1}{2}}}{\Gamma(\frac{n+1}{2})} \ \ \ \text{for $d$ even} \\ \\
    2 d^{\frac{n}{2}}\frac{\pi^{\frac{n+1}{2}}}{\Gamma(\frac{n+1}{2})} \ \ \ \text{for $d$ odd}
    \end{cases}.
\end{align}
\end{corollary}
In section \ref{volumechapter} we will need the explicit expression of $\mathrm{Vol}_n^{BW}(\SW_{n,d})$. This easily follows from formula \eqref{explicitVeronesevolume}, since $\SW_{n,d}=\SV_{n,d}\cup -\SV_{n,d}$ where $\SV_{n,d}=-\SV_{n,d}$ for $d$ odd and $\SV_{n,d}\cap (-\SV_{n,d})=\emptyset$ for $d$ even. Therefore
\begin{align}\label{explicitsphericalvolume}
    \mathrm{Vol}_n^{BW}(\SW_{n,d})= 
    2d^{\frac{n}{2}}\frac{\pi^{\frac{n+1}{2}}}{\Gamma(\frac{n+1}{2})}=d^{\frac{n}{2}}\mathrm{Vol}(S^n).
\end{align}\\

\begin{remark}\label{isometriesremark}
\noindent For every orthogonal matrix $R \in O(n+1)$ we have the following commutative diagram
\begin{align}\label{isometriesdiagram}
    \begin{tikzcd}[ampersand replacement=\&]
    (\mathds{RP}^n,g_{\mathds{RP}^n}) \arrow{d}[swap]{\nu_{n,d}} \arrow{r}{R} \& (\mathds{RP}^n,g_{\mathds{RP}^n}) \arrow{d}{\nu_{n,d}} \\
    \PV_{n,d} \arrow{r}{\rho(R)|_{\PV_{n,d}}} \& \PV_{n,d}
    \end{tikzcd},
\end{align}
since $\PV_{n,d}$ is an orbit for the action $\rho$ and is therefore preserved under $\rho(R)$. Moreover, by the invariance of the Bombieri--Weyl scalar product, $\rho(R)$ is an isometry of $(\mathds{RP}^N,g_{\mathds{RP}^N})$, and its restriction to $\PV_{n,d}$ defines an isometry of $\PV_{n,d}$. Therefore $\PV_{n,d}$ is an orbit for an isometric action of $O(n+1)$ over $\R[x_0,\dots,x_n]_{(d)}$ and these isometries of $\PV_{n,d}$ are induced by isometries of the ambient space. The same property also holds for $\SV_{n,d}$ considering the action on the sphere $S^N$. This simple observation, which is essentially due to $\PV_{n,d}$ and $\SV_{n,d}$ being orbits, will allow us to drastically simplify the computations we will carry out in Chapter \ref{reachchapter} and Section \ref{secondfundamentalchapter}. 
\end{remark}
\bigskip

\section{The reach of the spherical Veronese variety}\label{reachchapter}
In this chapter we provide an explicit computation for the reach of ${\SV}_{n,d} \hookrightarrow S^N$. The interest in this quantity relies on the fact it provides a lower bound for the $\varepsilon$'s of validity for Weyl's tube formula, as theorem \ref{Weyltheorem} shows. Remark that, since ${\SV}_{n,d}$ is compact, by theorem \ref{tubularneighbourhoodtheorem} we have $\rho({\SV}_{n,d})>0$. \\ \par
Recall that $\rho({\SV}_{n,d})=\min\{\rho_1({\SV}_{n,d}),\ \rho_2({\SV}_{n,d})\}$ and the expressions \eqref{rho_1} and \eqref{rho_2}. Using remark \ref{isometriesremark} we can prove the following.
\begin{lemma}\label{simplereachlemma}
The formula in \eqref{rho_1} simplifies to 
\begin{align}\label{simplerho_1}
    \bigl(\rho_1({\SV}_{n,d})\bigr)^{-1}=\sup\{\ \norm{\overset{..}\gamma(0)} \ | \ \gamma:(-\delta,\delta) \longrightarrow {\SV}_{n,d} \ \textrm{arclength geodesic in} \ {\SV}_{n,d}, \ \gamma(0)=x_0^d \  \},
\end{align}
that is to say the inner supremum in \eqref{rho_1} does not depend on $x \in {\SV}_{n,d}$. Moreover \eqref{simplerho_1} does not depend on the direction of $\overset{.}\gamma(0)$. Similarly, the formula in \eqref{rho_2} simplifies to
\begin{align}\label{simplerho_2}
    \rho_2({\SV}_{n,d})=\frac{1}{2}\ \inf  \{l(\gamma) \ \big| & \ \gamma:[a,b]\longrightarrow S^N \ \text{geodesic s.t.}\ \gamma(a)=x_0^d, \ \gamma(b) \in {\SV}_{n,d}, \\ \notag & \  \dot{\gamma}(a) \in N_{x_0^d}{\SV}_{n,d}, \ \dot{\gamma}(b) \in N_{\gamma(b)}\SV_{n,d} \}.
\end{align}
\end{lemma}
\begin{proof}
Consider an arclength geodesic $\gamma:(-\delta,\delta) \longrightarrow {\SV}_{n,d}$ with $\gamma(0)=p_1$ and pick another point $p_2 \in {\SV}_{n,d}$. By remark \ref{isometriesremark} there exists $R \in O(n+1)$ such that $\rho(R)p_1=p_2$. Recall that the image of a geodesic through an isometry is still a geodesic, hence $\tilde \gamma:=\rho(R)(\gamma)$ is an arclength geodesic with $\tilde \gamma(0)=\rho(R)(\gamma(0))=\rho(R)p_1=p_2$. Since $\rho(R)$ is also an isometry of the ambient space $S^N$, we have $\norm{\overset{..}\gamma(0)}=\norm{\overset{..}{\tilde \gamma}(0)}$. We just proved that given any two points in ${\SV}_{n,d}$, using the isometries $\rho(R)$ for $R \in O(n+1)$ we can transport any arclength geodesic passing through the first point into another arclength geodesic passing through the second point, preserving the norm of second derivatives. It follows that the expression in \eqref{rho_1} is independent of the specific point $x \in {\SV}_{n,d}$. Now observe that given any arclength geodesic $\gamma$ with $\gamma(0)=x_0^d$ and $\overset{.}\gamma(0)=v$, we can change the direction of $\overset{.}\gamma(0)$ through $\rho(R)$ for some $R \in O(n+1)$ with $x_0^d$ a fixed point (it is sufficient to choose a rotation $R$ such that $(1,0,\dots,0) \in S^n$ is in the axis of rotation), obtaining any other possible direction in $T_{x_0^d}{\SV}_{n,d}$ without changing $\norm{\overset{\cdot\cdot}\gamma(0)}$, for the same reason as above. It follows that \eqref{simplerho_1} does not depend on the specific direction of $\overset{.}\gamma(0)$. Since isometries preserve orthogonality and lengths, the second part of the lemma also follows in a similar way.
\end{proof}
\noindent The choice of $x_0^d$ in formulae \eqref{simplerho_1} and \eqref{simplerho_2} is motivated by convenience for computations only and we could have chosen any other point. \\
The first step to compute \eqref{simplerho_1} and \eqref{simplerho_2} for ${\SV}_{n,d} \hookrightarrow S^N$ is to understand tangent and normal spaces.
\begin{lemma}\label{tangentVeronese}
For $p \in {\SV}_{n,d}$ with $p=l^d$, where $l$ is a norm-$1$ linear form, we have
\begin{align*}
    T_p{\SV}_{n,d}=\big\langle\{l^{d-1}\lambda \ | \ \lambda \ \text{is a linear form orthogonal to $l$}\}\big\rangle.
\end{align*} 
\end{lemma}
\begin{proof}
Write $l(x)=a_0x_0+\dots+a_nx_n$ and set $a=(a_0,\dots,a_n) \in S^n$. Recalling that ${\SV}_{n,d}=\mathrm{im} (\widehat \nu_{n,d})$, a curve on ${\SV}_{n,d}$ can be expressed as the image of a curve on $S^n$ through $\widehat \nu_{n,d}$. Consider $b=(b_0,\dots,b_n) \in S^n$ such that $\langle a,b \rangle=0$. Then $\gamma(t)=(\cos t (a_0x_0+\dots+a_nx_n)+\sin t (b_0x_0+\dots+b_nx_n))^d$ is a curve in ${\SV}_{n,d}$ with $\gamma(0)=p$. Remark that the orthogonality condition is needed to ensure we are taking $d$--th power of a norm--$1$ form. We have 
\begin{align*}
    \frac{d}{dt}\gamma(t)\big|_{t=0} = d\ l^{d-1}(b_0x_0+\dots +b_nx_n),
\end{align*}
therefore $\big\langle\{l^{d-1}\lambda \ | \ \lambda \ \text{is a linear form orthogonal to $l$}\}\big\rangle \subset T_p{\SV}_{n,d}$. By dimension count, equality follows.
\end{proof}
Now that we have the ingredients we need to perform the computation of $\rho({\SV}_{n,d})$.
\begin{theorem}\label{reachtheorem}
For the reach of ${\SV}_{n,d} \hookrightarrow S^N$ we have 
\begin{align*}
    &\rho_1({\SV}_{n,d})=\frac{1}{\sqrt{2}}\sqrt{1+\frac{1}{d-1}}=\frac{1}{\sqrt{2}}+\frac{1}{2(d-1)\sqrt{2}}+\mathcal{O}\biggl(\frac{1}{d^2}\biggr), \\   &\rho_2({\SV}_{n,d})=\frac{\pi}{4}.
\end{align*}
Therefore the reach of ${\SV}_{n,d}$ is given by
\begin{align}\label{actualreach}
    \rho({\SV}_{n,d})=\min\{\rho_1({\SV}_{n,d}),\ \rho_2({\SV}_{n,d})\}=\begin{cases}
        \frac{\pi}{4} & 2\leq d\leq 5 \\
        \frac{1}{\sqrt{2}}\sqrt{1+\frac{1}{d-1}} & d \geq 6
    \end{cases}.
\end{align}
\end{theorem}
\begin{proof}
We begin with the computation of $\rho_1({\SV}_{n,d})$. By Proposition \ref{metricrelation} geodesics in ${\SV}_{n,d}$ can be realized as images through $\widehat \nu_{n,d}$ of geodesics in $S^n$. Moreover, thanks to Lemma \ref{simplereachlemma}, it is enough to consider geodesics passing through $x_0^d=(1,0,\dots,0)=\widehat \nu_{n,d}((1,0,\dots,0))$ at time $0$ and their direction plays no role, hence it is enough to consider the image of the geodesic in $S^n$ given by $\alpha(t)=x_0\cos(td^{-\frac{1}{2}})+x_1\sin(td^{-\frac{1}{2}})$ with $x_0$ being the point of coordinates $(1,0,\dots,0)$ and $x_1$ that of coordinates $(0,1,0,\dots,0)$. Explicitly we have $\alpha(t)=(\cos(td^{-\frac{1}{2}}),\sin(td^{-\frac{1}{2}}),0,\dots,0)$ and the corresponding geodesic in ${\SV}_{n,d}$ is given by
\begin{align*}
    \gamma(t):= (\widehat \nu_{n,d} \circ \alpha)(t)=\bigl(\cos^d(td^{-\frac{1}{2}}),\sqrt{d}\ \cos^{d-1}(td^{-\frac{1}{2}})\ \sin(td^{-\frac{1}{2}}),\dots,\sin^d(td^{-\frac{1}{2}})\bigr),
\end{align*}
with $\gamma(0)=x_0^d=(1,0,\dots,0)$ and $\norm{\dot{\gamma}(0)}=1$. Notice that the only components of $\gamma(t)$ which are not constantly zero are those corresponding to multi--indices $(\beta_0,\beta_1,0,\dots,0)$ with $\beta_0+\beta_1=d$. To compute $\norm{\overset{..}\gamma(0)}$, we need the second derivatives of the non--constantly zero components of $\gamma(t)$. These have the following expression for $k=0,\dots,d$:
\begin{align*}
    \binom{d}{k}^{\frac{1}{2}}\cos^k(td^{-\frac{1}{2}})\sin^{d-k}(td^{-\frac{1}{2}}).
\end{align*}
Computing second derivatives and evaluating at $t=0$ we find
\begin{align*}
    \overset{..}\gamma(0)=(-1,0,\sqrt{2}\frac{\sqrt{d(d-1)}}{d},0,\dots,0) .
\end{align*}
Since we are looking at $\SV_{n,d}$ as a submanifold of $S^N$, we first need to project this vector to the tangent space $T_{x_0^d}S^N$ and then compute its norm, obtaining
\begin{align*}
    \norm{\overset{..}{\gamma}(0)}_{S^N}=&\norm{\textrm{proj}_{T_{x_0^d}S^N}\big(-1,0,\sqrt{2}\frac{\sqrt{d(d-1)}}{d},0,\dots,0\big)}=\norm{\big(0,\sqrt{2}\frac{\sqrt{d(d-1)}}{d},0\dots,0\big)}\\= &\sqrt{2}\sqrt{\frac{d-1}{d}}.
\end{align*}
As a consequence we obtain the expression
\begin{align*}
    \rho_1(\Sigma_{n,d})=\frac{1}{\norm{\overset{..}{\gamma}(0)}}=\frac{1}{\sqrt{2}}\sqrt{\frac{d}{d-1}}=\frac{1}{\sqrt{2}}\sqrt{1+\frac{1}{d-1}}=\frac{1}{\sqrt{2}}+\frac{1}{2(d-1)\sqrt{2}}+\mathcal{O}\bigg(\frac{1}{d^2}\bigg)
\end{align*}
where in the last step we used the Taylor-MacLaurin expansion $\sqrt{1+x}=1+\frac{x}{2}+\mathcal{O}(x^2)$ for $x=\frac{1}{d-1}$.
\begin{remark}\label{remark:lower}Notice that, since $\rho_1(\SV_{n,d})=\frac{1}{\sqrt{2}}\sqrt{1+\frac{1}{d-1}}$, then $\rho_1(\SV_{n,d})>\frac{1}{\sqrt{2}}$ for all $n,d$.
\end{remark}
For $\rho_2({\SV}_{n,d})$, by Lemma \ref{simplereachlemma} it is enough to consider geodesics $\gamma(\theta)$ in $S^N$ starting at $x_0^d$ at time $\theta=0$. By Lemma \ref{tangentVeronese} and recalling that the normal space at a point $p \in {\SV}_{n,d}$ is the orthogonal complement of $T_p{\SV}_{n,d}$ inside $T_pS^N$, we have $N_{x_0^d}{\SV}_{n,d}=\big\langle\bigl\{\binom{d}{\alpha}^{\frac{1}{2}}x_0^{\alpha_0}\dots x_n^{\alpha_n} \ | \ \alpha_0<d-1 \bigr\}\big\rangle$. Pick a vector $w \in N_{x_0^d}{\SV}_{n,d}$  and let $\gamma_w(\theta)$ be the geodesic in $S^N$ with $\gamma_w(0)=x_0^d$ and $\dot{\gamma}_w(0)=w$, i.e. 
\begin{align*}
    \gamma_w(\theta)=x_0^d\ \cos\bigl(\theta\norm{w}\bigr)+\frac{w}{\norm{w}}\ \sin\bigl(\theta\norm{w}\bigr).
\end{align*}
The goal now is to understand when $\gamma_w$ meets again ${\SV}_{n,d}$ orthogonally. The first step is to find for which $b=(b_0,\dots,b_n) \in S^n$ and $\theta$ we at least have a solution to the equation
\begin{align*}
    (b_0x_0+\dots + b_nx_n)^d = x_0^d\ \cos\bigl(\theta\norm{w}\bigr)+\frac{w}{\norm{w}}\ \sin\bigl(\theta\norm{w}\bigr).
\end{align*}
On the right hand side, we expand $w$ as $w=\sum_{\alpha_0 < d-1}{w_{\alpha}\binom{d}{\alpha}^{\frac{1}{2}}x^{\alpha}}$, while we expand the left hand side as $\sum_{\alpha}{\binom{d}{\alpha}b^{\alpha}x^{\alpha}}$ with multi--indices $\alpha=(\alpha_0,\dots,\alpha_n)\in \Z_{\geq0}^{n+1}$ such that $\alpha_0+\dots+\alpha_n=d$. Hence, expanding it further in the Bombieri--Weyl basis, we get the equation
\begin{align}\label{equating}
b_0^dx_0^d + \sqrt{d}\ x_0^{d-1}\biggl(\sum_{i=1}^n{\sqrt{d}\ b_0^{d-1}b_ix_i}\biggr)+\sum_{\alpha_0 < d-1}{\binom{d}{\alpha}b^{\alpha}x^{\alpha}} = \\ \notag x_0^d\ \cos\bigl(\theta\norm{w}\bigr)+\sin\bigl(\theta\norm{w}\bigr)\sum_{\alpha_0<d-1}{\binom{d}{\alpha}^{\frac{1}{2}}\frac{w_{\alpha}}{\norm{w}}x^{\alpha}}.
\end{align}
Equating corresponding coefficients we get
\begin{align*}
    \begin{cases}
    b_0^d=\cos\bigl(\theta\norm{w}\bigr) & \\
    b_0^{d-1}b_i=0 & \forall \  i=1,\dots,n
\end{cases}.
\end{align*}
Now two cases can occur:
\begin{itemize}
    \item if $b_0 \neq 0$, the second equation above implies that $b_i=0$ for $i=1,\dots,n$, hence $b=(b_0,0,\dots,0)$. Since $b \in S^n$ it follows that $b_0=\pm1$. If $b_0=1$, then the meeting point is again $x_0^d$ and $\gamma_w$ comes back to it for $\theta=\frac{2\pi}{\norm{w}}$, and the same happens if $b_0=-1$ and $d$ is even. If $b_0=-1$ and $d$ is odd, then the meeting point corresponds to $-x_0^d$, and the meeting time is $\theta=\frac{\pi}{\norm{w}}$.
    \item if $b_0=0$ we get $\cos(\theta\|w\|)=0$ and $\gamma_w$ may meet ${\SV}_{n,d}$ at $(b_0x_0+\dots+b_nx_n)^d$ for $\theta=\frac{\pi}{2\norm{w}}$ or $\theta=\frac{3\pi}{2\norm{w}}$.
\end{itemize}
Now that we know when the curve intersects again ${\SV}_{n,d}$, we need to understand when it does that orthogonally. Notice that up to now we have used only some of the equations arising from \eqref{equating}: we now use the others to impose the orthogonality condition. \\
We look at the case $b_0=0$ and $\theta=\frac{\pi}{2\norm{w}}$. If we could find some $b=(0,b_1,\dots,b_n) \in S^n$ and $w \in N_{x_0^d}{\SV}_{n,d}$ such that $\gamma_w$ meets ${\SV}_{n,d}$ orthogonally at 
$(b_0x_0+\dots+b_nx_n)^d$ for $\theta=\frac{\pi}{2\norm{w}}$, then by the previous computation no other curve satisfying the conditions in \eqref{simplerho_2} could have length less than this one. Fix the following notations $\tilde b := (b_1,\dots,b_n)$, $\tilde x:=(x_1,\dots,x_n)$ and $\tilde \alpha:=(\alpha_1,\dots,\alpha_n)$. We have $\norm{\tilde b}=1$ and for $\theta = \frac{\pi}{2\norm{w}}$ equation \eqref{equating} becomes
\begin{align}\label{simplifiedequating}
    \sum_{\alpha_1+\dots+\alpha_n=d}{\binom{d}{\tilde \alpha}\tilde b ^{\tilde \alpha}\tilde x ^{\tilde \alpha}}= \sum_{\alpha_0 < d-1}{\binom{d}{\alpha}^{\frac{1}{2}}\frac{w_{\alpha}}{\norm{w}}x^{\alpha}},
\end{align}
Equating corresponding coefficients, we get $w_{\alpha}=0$ for each $\alpha=(\alpha_0,\dots,\alpha_n)$ such that $\alpha_0\neq 0$, while for the other multi--indices we get $\binom{d}{\tilde \alpha}^{\frac{1}{2}}\tilde b^{\tilde \alpha}=\frac{w_{\tilde \alpha}}{\norm{w}}$. These equations admit a solution and therefore a curve $\gamma_w$ with the properties described above exists. \\ Since we are assuming that $\gamma_w\bigl(\frac{\pi}{2\|w\|}\bigr)=(b_1x_1+\dots+b_nx_n)^d$, by Lemma \ref{tangentVeronese} we have $T_{\gamma_w\bigl(\frac{\pi}{2\norm{w}}\bigr)}{\SV}_{n,d} = \big\langle \tilde b^{d-1}\lambda \ | \ \lambda \ \text{is a linear form orthogonal to}\  \tilde b \big\rangle$.
The tangent vector to the curve $\gamma_w$ at the meeting point is given by
\begin{align*}
    \dot{\gamma}_w\biggl(\frac{\pi}{2\norm{w}}\biggr)=-x_0^d\norm{w}.
\end{align*}
Hence, since the monomials of the Bombieri--Weyl basis are mutually orthogonal and $x_0^d$ is not among those spanning the tangent space at the meeting point, $\gamma_w$ comes back to ${\SV}_{n,d}$ orthogonally for any choice of $w \in N_{x_0^d}{\SV}_{n,d}$ and $b \in S^n$ satisfying the conditions given by \eqref{simplifiedequating}. Moreover, the length of such a curve is independent of $w$ and always equal to $\frac{\pi}{2}$. Hence we get that $\rho_2({\SV}_{n,d})=\frac{\pi}{4}$.\\
Noticing that $\frac{1}{\sqrt{2}}\sqrt{1+\frac{1}{d-1}} < \frac{\pi}{4}$ as soon as $d \geq 6$, the proof is complete.
\end{proof}
Remark that the same results hold true for $\SW_{n,d} \hookrightarrow S^N$. Before moving on, let us comment on the meaning of the quantity $\rho_2({\SW}_{n,d})$ in our context. Recall once again our interpretation of ${\SW}_{n,d} \hookrightarrow S^N$ as the set of symmetric rank--one tensors among norm--$1$ ones. Since ${\SW}_{n,d}$ is compact and in particular closed, for every point in $S^N$ there will be a point in ${\SW}_{n,d}$ minimizing the distance between the chosen point and ${\SW}_{n,d}$. The point realizing the minimum need not be unique and indeed in general it is not. From the tensor point of view, fixed $p$ every distance minimizing point in ${\SW}_{n,d}$ provides a best rank--one approximation of the tensor represented by $p$. Since $S^N$ is compact, it is geodesically complete, and therefore there exists distance minimizing geodesics joining $p$ with each of the points minimizing the distance and each of the geodesics meet ${\SW}_{n,d}$ orthogonally. From this observation we obtain that a symmetric norm--$1$ tensor admits more than one best rank--one approximation if and only if the corresponding point in $S^N$ admits more than one distance minimizing geodesic orthogonal to ${\SW}_{n,d}$. It follows that the injectivity of the normal exponential map on $N^{\varepsilon}{\SW}_{n,d}$ ensures that every tensor represented by a point in the image admits a unique best rank--one approximation, since it will be joined to ${\SW}_{n,d}$ by a unique distance minimizing orthogonal geodesic. 

Given this, we can restate the result of Theorem \ref{reachtheorem} about $\rho_2({\SW}_{n,d})$ in the following way.
\begin{corollary}\label{bestrankonecorollary}
Every symmetric tensor $p$ at distance less than $\sin(\rho(V_{n,d})) \|p\|_{\mathrm{BW}}$ from rank--one admits a unique best rank--one approximation.\end{corollary}
A consequence of this result is that the probability that a symmetric tensor admits a unique best rank--one approximation is bounded below by the normalized volume of the tubular neighbourhood of radius $\rho_{n,d}:=\rho(V_{n,d})$ on the sphere, that is to say,
\begin{align}\label{probabilityestimate}
    \mathbb{P}\big(\text{symmetric norm--$1$ tensor s.t.}\ \exists ! \ \text{best rank--one approximation}\big) \geq \frac{\mathrm{Vol}\big(\mathcal{U}({\SW}_{n,d},\rho_{n,d})\big)}{\mathrm{Vol}(S^N)}.
\end{align}

In Section \ref{volumechapter} we will apply \eqref{sphericalWeyl} to compute an exact formula for the volume of a tubular $\varepsilon$--neighbourhood $\mathrm{Vol}(\mathcal{U}({\SV}_{n,d},\varepsilon))$. We stress again that the reach computed in Theorem \ref{reachtheorem} gives a lower bound to the $\varepsilon$'s of validity of the formula we will find: we are guaranteed that it gives the correct result for any $\varepsilon < \rho({\SV}_{n,d})$. 
\bigskip

\section{The volume of the tubular neighbourhood}
\subsection{The second fundamental form of the spherical Veronese surface}\label{secondfundamentalchapter}
By Remark \ref{isometriesremark}, since we have an isometric transitive action of $O(n+1)$ on ${\SV}_{n,d}$ by restrictions of isometries of $S^N$, the extrinsic geometry of ${\SV}_{n,d}\hookrightarrow S^N$ is invariant under this action. Therefore, if we compute the second fundamental form at a specific point of ${\SV}_{n,d}$ we automatically know it at every point. We will now carry out the computation using the point $x_0^d \in {\SV}_{n,d}$ for simplicity. \\ \par 
By Remark \ref{secondfundformremark} to compute the second fundamental form of ${\SV}_{n,d}$ at $x_0^d$ along a normal direction $\eta$ it is enough to choose a local parametrization around $x_0^d$, compute its second derivatives and take their (Bombieri--Weyl) scalar product in $\R^{N+1}$ with $\eta$. This way we will obtain the matrix representing the Weingarten operator at $x_0^d$ with respect to the basis of $T_{x_0^d}{\SV}_{n,d}$ given by the derivatives of the chosen parametrization. \\ 
Consider the projection on $S^n$ from the tangent plane at $(1,0\dots,0)$, giving a parametrization of the upper hemisphere. Composing it with the Veronese map $\widehat \nu_{n,d}$ we obtain a parametrization $\varphi_{n,d}$ of the part of ${\SV}_{n,d}$ contained in the upper hemisphere of $S^N$, explicitly given by
\begin{align*}
    \varphi_{n.d}: \ \mathds{R}^n \ & \longrightarrow \ U \subset {\SV}_{n,d} \\ \notag
    a=(a_1,\dots,a_n) & \longmapsto \biggl(\frac{x_0+a_1x_1+\dots+a_nx_n}{(1+\|a\|^2)^{\frac{1}{2}}}\biggr)^d.
\end{align*}
Since $\varphi_{n,d}^{-1}(x_0^d)=(0,\dots,0)$, we have to compute the first and second derivatives of $\varphi_{n,d}$ at the origin. We obtain the following expressions
\begin{align}
\frac{\partial \varphi_{n,d}}{\partial a_i}(a)\bigg|_{a=0}&=dx_0^{d-1}x_i, \label{firstderivatives}\\ 
    \frac{\partial^2\varphi_{n,d}}{\partial a_i\partial a_j}(a)\bigg|_{a=0}&=-\delta_{ij}(d x_0^d)+d(d-1)x_0^{d-2}x_ix_j.
\end{align}
For $i=1,\dots,n$ denote by $e_i=\sqrt{d}x_0^{d-1}x_i$ the orthonormal vectors in the Bombieri--Weyl basis with power $d-1$ on $x_0$. By \eqref{firstderivatives} the basis of $T_{x_0^d}{\SV}_{n,d}$ given by the first derivatives of the parametrization is $\{\sqrt{d}e_i \ | \ i=1,\dots,n\}$. Instead of using this basis, we compute the matrix representing the Weingarten operator with respect to the orthonormal basis $\{e_i\}_{i=1,\dots,n}$ along a normal direction $\eta \in N_{x_0^d}{\SV}_{n,d}$. Denoting by $L_{\eta}=(L_{\eta,ij})_{i,j=1,\dots,n}$ this matrix, by Remark \ref{Weingartenmatrixremark} we have
\begin{align}\label{weingartenmidstep}
    L_{\eta,ij}&=H_{\eta}(e_i,e_j) \ = \ \frac{1}{d}\ H_{\eta}\biggl(\sqrt{d}e_i,\sqrt{d}e_j\biggr) = \frac{1}{d}\ H_{\eta}\biggl(\frac{\partial \varphi_{n,d}}{\partial a_i}(a)\bigg|_{a=0},\frac{\partial \varphi_{n,d}}{\partial a_j}(a)\bigg|_{a=0}\biggr) = \\ \notag & = \ \frac{1}{d}\ \big\langle \ \frac{\partial^2 \varphi_{n,d}}{\partial a_i \partial a_j}(a)\bigg|_{a=0},\eta \ \big\rangle_{\mathds{R}^{N+1}} \ = \ \frac{1}{d} \ \big\langle -\delta_{ij}(dx_0^d)+d(d-1)x_0^{d-2}x_ix_j, \eta \big\rangle_{\mathds{R}^{N+1}}.
\end{align}
By Lemma \ref{tangentVeronese} we have $N_{x_0^d}{\SV}_{n,d}=\big\langle\bigl\{\binom{d}{\alpha}^{\frac{1}{2}}x_0^{\alpha_0}\dots x_n^{\alpha_n} \ | \ \alpha_0<d-1 \bigr\}\big\rangle$ and we can expand $\eta=\sum_{\alpha_0\leq d-2}\eta_{\alpha}\binom{d}{\alpha}^{\frac{1}{2}}x^{\alpha}$. Then from \eqref{weingartenmidstep}, recalling that everything is expressed in terms of an orthonormal basis, we obtain 
\begin{align} \label{Weingartendiagonal}
    L_{\eta,ii}&=\sqrt{2\bigl(\frac{d-1}{d}\bigr)}\ \eta_{d-2,0,,\dots,2,\dots,0}, \\ \label{Weingartenoffdiagonal}
    L_{\eta,ij}&=\sqrt{\frac{d-1}{d}}\ \eta_{d-2,\dots,1,\dots,1,\dots,0} \ \ \text{for $i\neq j$}.
\end{align}
Consider the following orthogonal direct sum decomposition of $N_{x_0^d}{\SV}_{n,d}$
\begin{align}\label{normaldecomposition}
    N_{x_0^d}{\SV}_{n,d}=\big\langle\bigg\{\binom{d}{\alpha}^{\frac{1}{2}}x_0^{d-2}x_ix_j \ | \ i,j=1,\dots,n\bigg\}\big\rangle \oplus \big\langle\bigg\{\binom{d}{\alpha}^{\frac{1}{2}}x^{\alpha} \ | \ \alpha_0 < d-2\bigg\}\big\rangle =: W \oplus P.
\end{align}
We define a map from $W$ to $\R[x_1,\dots,x_n]_{(2)}$ by setting 
\begin{align*}
    \binom{d}{\alpha}^{\frac{1}{2}}x_0^{d-2}x_ix_j \longmapsto \binom{2}{(\alpha_i,\alpha_j)}^{\frac{1}{2}}x_ix_j
\end{align*}
and extending by linearity. Since we are mapping an orthonormal basis for $W$ with the induced Bombieri--Weyl product to an orthonormal basis of $\R[x_1,\dots,x_n]_{(2)}$ with its own Bombieri--Weyl product, this defines a linear isometry. Composing it with the inverse of the isomorphism we described in \eqref{isometryFrobKost} we get a linear isometry of $W$ with $\mathrm{Sym}(n,\R)$ and therefore a correspondence between the associated Gaussian probability distributions. A direct consequence of this linear isometry, of the discussion in section \ref{GOEKostlansection} about $\mathrm{GOE}(n)$ matrices and formulae \eqref{Weingartendiagonal} and \eqref{Weingartenoffdiagonal}, is the following theorem.
\begin{theorem}\label{normaldecompositiontheorem}
Consider the decomposition $N_{x_0^d}{\SV}_{n,d} = W \oplus P$ given in \eqref{normaldecomposition}. Then the following statements hold: 
\begin{enumerate}
    \item $L_{\eta}=0$ for every $\eta \in P$;
    \item if we pick $\eta \in W$ Gaussian w.r.t the Bombieri--Weyl metric, then the distribution of the Weingarten operator at $x_0^d$ along $\eta$ is  $L_{\eta} \sim \sqrt{2}\biggl(\frac{d-1}{d}\biggr)^{\frac{1}{2}}\mathrm{GOE}(n)$.
\end{enumerate}  
\end{theorem}
\begin{remark}Notice that Theorem \ref{thm:decointro} from the Introduction follows immediately from Theorem \ref{normaldecompositiontheorem} using the fact that for every $p\in V_{n,d}$ there is a linear isometry $\tau:S^N\to S^N$ such that $\tau(V_{n,d})=V_{n,d}$ and $\tau (p)=x_0^d.$
\end{remark}
Theorem \ref{normaldecompositiontheorem} gives a full description of the extrinsic geometry of ${\SV}_{n,d} \hookrightarrow S^N$ in terms of random matrices. From the computational point of view, it allows reducing integrals on the normal bundle of ${\SV}_{n,d}$ of quantities related to the second fundamental form to expected values of quantities related to $\mathrm{GOE}(n)$ matrices. In the next section, we will use this description to explicitly compute the integrals appearing in Weyl's tube formula \eqref{sphericalWeyl}, thus obtaining the curvature coefficients of the embedding $V_{n,d} \hookrightarrow S^N$. \\

\subsection{The curvature coefficients}\label{volumechapter}
All this section will be dedicated to proving the following.
\begin{theorem}\label{finalvolumetheorem}
Let ${\SW}_{n,d} \hookrightarrow S^N$ be the spherical Veronese variety and \, $\mathcal{U}({\SW}_{n,d},\varepsilon)$ be defined as in \eqref{tubularneighbourhooddef}. If $\varepsilon < \rho({\SW}_{n,d})$, the following formula holds:
\begin{align}\label{finalformulasimplified}
     \mathrm{Vol}(\mathcal{U}({\SW}_{n,d},\varepsilon))=\sum_{0\le j \le n, \text{\ j even}}(-1)^{\frac{j}{2}}d^{\frac{n}{2}}\biggl(\frac{d-1}{d}\biggr)^{\frac{j}{2}}
    \frac{2^{n+2-j}\pi^{\frac{N}{2}}\Gamma\left(\frac{n}{2}+1\right)}{\Gamma\left(\frac{j}{2}+1\right)\Gamma(n+1-j)\Gamma\left(\frac{N+j-n}{2}\right)}J_{N,N-n+j}(\varepsilon),
\end{align}
where for $0\leq j \leq n$, $j$ even, the functions $J_{N,N-n+j}$ are given by \eqref{Jspherical}.
\end{theorem}
\noindent Comparing \eqref{finalformulasimplified} to Weyl's tube formula \eqref{implicitWeyl}, we obtain the following corollary.
\begin{corollary}
The curvature coefficients of the spherical Veronese variety ${\SW}_{n,d} \hookrightarrow S^N$ are as follows:
\begin{align*}
     K_{N-n+j}({\SW}_{n,d})=(-1)^{\frac{j}{2}}d^{\frac{n}{2}}\biggl(\frac{d-1}{d}\biggr)^{\frac{j}{2}}
    \frac{2^{n+2-j}\pi^{\frac{N}{2}}\Gamma\left(\frac{n}{2}+1\right)}{\Gamma\left(\frac{j}{2}+1\right)\Gamma(n+1-j)\Gamma\left(\frac{N+j-n}{2}\right)},
\end{align*}
for $0 \leq j \leq n$, $j$ even, and $K_{N-n+j}(\SW_{n,d})=0$ otherwise.
\end{corollary} \smallskip
In order to prove \ref{finalvolumetheorem} we start from Weyl's tube formula \eqref{sphericalWeyl} applied to $\SV_{n,d}\hookrightarrow S^N$. As we already noticed, Remark \ref{isometriesremark} implies that the Weingarten operator looks the same at every point. It follows that in this case the integrand in \eqref{sphericalWeyl} does not depend on $p \in {\SV}_{n,d}$ and we obtain 
\begin{align}\label{simp1integral}
    \mathrm{Vol}\bigl(\mathcal{U}({\SV}_{n,d},\varepsilon)\bigr)=\mathrm{Vol}({\SV}_{n,d})\int_{t=0}^{\tan{\varepsilon}}{\int_{\eta \in S(N_{x_0^d}{\SV}_{n,d})}{\frac{t^{N-n-1}\det(I_n-t L_{\eta})}{(1+t^2)^{\frac{N+1}{2}}}}\, d\eta \, dt},
\end{align}
where we remark that $N-n$ is the codimension of ${\SV}_{n,d} \hookrightarrow S^N$ and $\mathrm{Vol}({\SV}_{n,d})$ is given by \eqref{explicitVeronesevolume}. \\

Given the decomposition in \eqref{normaldecomposition}, we have $S(N_{x_0^d}{\SV}_{n,d})=S(W \oplus P)$, where $\mathrm{dim}(S(N_{x_0^d}{\SV}_{n,d}))=N-n-1$ and $\mathrm{dim}(W)=\mathrm{dim}(\mathrm{Sym}(n,\R))=\frac{n(n+1)}{2}$. Notice that if $d=2$ we have $N_{x_0^d}{\SV}_{n,d}=W$. If $d>2$ we parametrize $S(W \oplus P)$ as in \eqref{parametrizationlemma}, where here we use $m=N-n-1$ and $k=\frac{n(n+1)}{2}-1$. With the same notation of section \ref{sectionlemma}, for $\sigma \in S(W)$ and $z \in \overset{\circ}{D}{}(P)$, if $\varphi(\sigma,z)=\eta \in S(W\oplus P)$, we have that $\sqrt{1-\abs{z}^2}\iota(\sigma)$ will be the component of $\eta$ along $W$, while $z$ itself will be the component along $P$. We also apply the linear isometry discussed in the previous section to change variable from $\sigma \in S(W)$ to $Q \in S(\mathrm{Sym}(n,\R))=S^{\frac{n(n+1)}{2}-1}$. \\

It is clear by its definition that the Weingarten operator is linear in the normal vector argument: given an isometric embedding $M \hookrightarrow \overline{M}$, for every $p\in M$, $\eta$, $\xi \in N_{p}M$ and $a$, $b \in \R$, we have $L_{a\eta+b\xi}=a L_{\eta}+b L_{\xi}$. Therefore for $\eta=\varphi(Q,z) \in S(W\oplus P)$ we have
\begin{align}\label{secondformsplitting}
     L_{\eta}=L_{\sqrt{1-\abs{z}^2}Q+z}=\sqrt{1-\abs{z}^2}L_{Q}+L_z= \sqrt{1-\abs{z}^2}L_{Q}.
\end{align}
Applying Lemma \ref{integrationlemma} to \eqref{simp1integral} and using \eqref{secondformsplitting} the integral becomes
\begin{align}\label{simp2integral}
    \mathrm{Vol}\bigl(\mathcal{U}({\SV}_{n,d},\varepsilon)\bigr)=&\mathrm{Vol}({\SV}_{n,d})\int_{t=0}^{\tan{\varepsilon}}{\int_ {S^{\frac{n(n+1)}{2}-1}}{\int_{D^{N-n-\frac{n(n+1)}{2}}}\bigg[{\frac{t^{N-n-1}}{(1+t^2)^{\frac{N+1}{2}}}}}} \times \\ \notag &\times \det(I_n-t\sqrt{1-\abs{z}^2} L_Q) (1-\abs{z}^2)^{\frac{n(n+1)}{4}-1}\bigg] \ dz \ dS(Q) \ dt,
\end{align}
where $dz$ is a short notation for $\mathrm{vol}_{D^{N-n-\frac{n(n+1)}{2}}}$ and $dS(Q)$ is a short notation for $\mathrm{vol}_{S^{\frac{n(n+1)}{2}-1}}$, with the convention that for $d=2$ the integral over $D^0$ is set to $1$. The only non-explicit term in \eqref{simp2integral} is the one involving the determinant. Recall that by Theorem \ref{normaldecompositiontheorem}, if $Q \in \mathrm{Sym}(n,\R)$ is a random $\mathrm{GOE}(n)$ matrix, then $L_Q$ is a random matrix distributed as $\sqrt{2}\big(\frac{d-1}{d}\big)^{\frac{1}{2}}\mathrm{GOE}(n)$. Set $\tau:=t\sqrt{2}\big(\frac{d-1}{d}\big)^{\frac{1}{2}}$. We have the expansion
\begin{align}\label{detexpansion}
    \mathrm{det}\big(I_n - \tau\sqrt{1-\abs{z}^2}Q\big)=\sum_{j=0}^n
(-1)^j\tau^j(1-\abs{z}^2)^{\frac{j}{2}}g_j(Q),
\end{align}
where $g_j(Q)$ are homogeneous polynomials of degree $j$ in the coefficients of $Q$ for $j=1,\dots,n$ and $g_0(Q)=1$. Substituting \eqref{detexpansion} into \eqref{simp2integral} in the integral splits as
\begin{align}\label{simp3integral}
    \mathrm{Vol}(\mathcal{U}({\SV}_{n,d},\varepsilon))=& \ \mathrm{Vol}({\SV}_{n,d})\mathlarger{\sum}_{j=0}^n{(-1)^j2^{\frac{j}{2}}\biggl(\frac{d-1}{d}\biggr)^{\frac{j}{2}}\biggl(\int_0^{\tan\varepsilon}{\frac{t^{N-n-1+j}}{(1+t^2)^{\frac{N+1}{2}}}\ dt}\biggr)} \times \\ \notag \times & \biggl(\int_{D^{N-n-\frac{n(n+1)}{2}}}{(1-\abs{z}^2)^{\frac{n(n+1)}{4}-1+\frac{j}{2}}\ dz}\biggr) \times \\ \notag \times & \biggl(\int_{S^{\frac{n(n+1)}{2}-1}}{g_j(Q) \ dS(Q)}\biggr),
\end{align}
where the first term is the integral of a rational function in $t$, while the second one is a \lq\lq polynomial\rq\rq \ in $\abs{z}$. \\
Remark that since $g_j$ are homogeneous polynomials, we have $g_j(Q)=\norm{Q}^jg_j(\frac{Q}{\norm{Q}})$. Recalling expression \eqref{GOEdensity} we have
\begin{align}\label{expectationintegral}
    \underset{Q \in \mathrm{GOE}(n)}{\mathds{E}} g_j(Q) \ & = \ \frac{1}{(2\pi)^{\frac{n(n+1)}{4}}}\int_{\mathrm{Sym}(n,\R)}{\norm{Q}^j g_j\biggl(\frac{Q}{\norm{Q}}\biggr)e^{-\frac{\norm{Q}^2}{2}}dQ} \ = \\ \notag & = \frac{1}{(2\pi)^{\frac{n(n+1)}{4}}}\biggl(\int_0^{+\infty}{\rho^{\frac{n(n+1)}{2}-1+j}e^{-\frac{\rho^2}{2}}d\rho}\biggr)\biggl(\int_{S^{\frac{n(n+1)}{2}-1}}{g_j(\tilde Q) \ dS\big(\tilde Q\big)}\biggr).
\end{align}
From \eqref{expectationintegral} we obtain
\begin{align}\label{g_jbeforesub}
    \int_{S^{\frac{n(n+1)}{2}-1}}{g_j(Q)\ dS(Q)} \ = \ \frac{\underset{Q\in \mathrm{GOE}(n)}{\mathds{E}}[g_j(Q)] \ (2\pi)^{\frac{n(n+1)}{4}}}{\int_0^{+\infty}{\rho^{\frac{n(n+1)}{2}-1+j}e^{-\frac{\rho^2}{2}}d\rho}}.
\end{align}
By linearity of expectation and the expansion $\mathrm{det}(I_n-\lambda Q)=\sum_{j=0}^n (-1)^j\lambda^j g_j(Q)$, to compute the expectation of $g_j(Q)$ it is enough to compute that of $\mathrm{det}(I_n-\lambda Q)$ for $Q \in \mathrm{GOE}(n)$ and look at the homogeneous part of degree $j$ in $\lambda$. This procedure gives us the explicit expression for \eqref{g_jbeforesub} 
\begin{align}\label{explicitthirdpiece}
    \int_{S^{\frac{n(n+1)}{2}-1}}{g_j(Q)\ dS(Q)} = \frac{(-1)^{\frac{j}{2}}(2\pi)^{\frac{n(n+1)}{4}}\frac{j!}{(\frac{j}{2})!}\binom{n}{j}}{2^j\int_0^{+\infty}{\rho^{\frac{n(n+1)}{2}-1+j}e^{-\frac{\rho^2}{2}}d\rho}} \quad \text{if $0\le j\le n$, $j$ even}
\end{align}
and $0$ otherwise, see Appendix \ref{expectationappendix} for a proof of this result. By standard computations involving Gamma and Beta functions, one can show that the following identities hold
\begin{align}
    &\label{rhointegral} \int_0^{+\infty}\rho^{\frac{n(n+1)}{2}+j-1}e^{-\frac{\rho^2}{2}}\,d\rho \ = \ 2^{\frac{n(n+1)}{4}+\frac{j}{2}-1}\,\Gamma\biggl(\frac{1}{4}(n^2+n+2j)\biggr), \\ \label{diskintegral}
    &\int_{D^{N-n-\frac{n(n+1)}{2}}}\bigl(1-\abs{z}^2\bigr)^{\frac{n(n+1)}{4}-1+\frac{j}{2}}\ dz \ = \ \pi^{\frac{2N-n^2-3n}{4}}\,\frac{\Gamma\bigl(\frac{1}{4}(n^2+n+2j)\bigr)}{\Gamma\bigl(\frac{1}{2}(N-n+j)\bigr)},
\end{align}
where we notice that for $d=2$ \eqref{diskintegral} gives $1$, agreeing with our convention. Substituting \eqref{explicitthirdpiece}, \eqref{rhointegral} and \eqref{diskintegral} into \eqref{simp3integral} and using the duplication formula for the gamma function, we obtain the explicit expression of $\mathrm{Vol}\bigl(\mathcal{U}({\SV}_{n,d},\varepsilon)\bigr)$. Finally, recalling that $\SW_{n,d}=\SV_{n,d}\cup -\SV_{n,d}$ and using formula \eqref{explicitVeronesevolume} to express $\mathrm{Vol}(\SV_{n,d})$, the proof of Theorem \ref{finalvolumetheorem} is complete.

\appendix

\section{}\label{tubularneighbourhoodproof}
\subsection*{Proof of the Tubular Neighbourhood Theorem}
We will use the same notation of Section \ref{tubularsection}. Throughout the proof, we will identify $M$ with the zero section in $NM$. We start by computing the differential  $d_{(x,0)}(\textrm{exp}|_{NM}):T_{(x,0)}(NM)\longrightarrow T_x\overline{M}$ of $\mathrm{exp}|_{NM}$ at $(x,0) \in NM$ for any $x \in M$. Notice that $\textrm{dim}(T_{(x,0)}(NM))=\textrm{dim}(T_x\overline{M})$, hence surjectivity is enough to have a linear isomorphism. Denote by $\gamma_{(p,v)}$ the unique geodesic on $\overline{M}$ such that $\gamma_{(p,v)}(0)=p$ and $\dot{\gamma}_{(p,v)}(0)=v$. Let $y \in T_x\overline{M}$. Since $T_x\overline{M}=T_xM \oplus N_xM$, we can decompose $y$ as $y=y_1+y_2$ with $y_1 \in T_xM$ and $y_2 \in N_xM$. Then there exists $\sigma_1:(-\delta,\delta)\longrightarrow M$ such that $\sigma_1(0)=x$ and $\dot{\sigma}_1(0)=y_1$. Define a curve $\sigma:(-\delta,\delta)\longrightarrow NM$ by $\sigma(t)=(\sigma_1(t),0)\in NM$. We have $\sigma(0)=(x,0)$ and $\dot{\sigma}(0)=(y_1,0)\in T_{(x,0)}(NM)$ and it follows that
\begin{align*}
    d_{(x,0)}(\textrm{exp}|_{NM})(y_1,0)&=\frac{d}{dt}\textrm{exp}\bigl(\sigma(t)\bigr)\bigg|_{t=0}=\ y_1,
\end{align*}
proving that $T_xM$ is contained in the image of $d_{(x,0)}(\textrm{exp}|_{NM})$. Now take $y_2 \in N_xM$ and define a curve $\alpha:(-\delta,\delta)\longrightarrow NM$ by $\alpha(t)=(x,ty_2)$. Then $\alpha(0)=(x,0)$ and $\dot{\alpha}(0)=(0,y_2)$ and it follows that
\begin{align*}
    d_{(x,0)}(\textrm{exp}|_{NM})(0,y_2))&=\frac{d}{dt} \textrm{exp}\bigl(\alpha(t)\bigr)\bigg|_{t=0}=\ y_2,
\end{align*}
proving that also $N_xM$ is contained in the image of $d_{(x,0)}(\textrm{exp}|_{NM})$. By linearity of the differential, we obtain surjectivity and therefore $d_{(x,0)}(\textrm{exp}|_{NM})$ is an isomorphism.\\  As a consequence for every $x \in M$ there exists an open neighbourhood $W_x$ of $(x,0)$ in $NM$ such that the rank of the differential $d_{(q,v)}(\textrm{exp}|_{NM})$ is maximal for every $(q,v) \in W_x$. Up to shrinking the neighbourhood, we can assume that $W_x=\bigl(U_x \times B(0,\varepsilon_x)\bigr)\cap NM$ where $U_x$ is an open neighbourhood of $x \in M$, $B(0,\varepsilon_x)$ denotes the ball of radius $\varepsilon_x$ centered at the origin in $T_x\overline{M}$ and $\textrm{exp}|_{W_x}$ is an embedding. By compactness we have a finite covering of $M$ $\{U_{x_1},\dots,U_{x_r}\}$ for some $x_1,\dots,x_r \in M$. Choosing $\varepsilon:=\min\{\varepsilon_{x_1},\dots,\varepsilon_{x_r}\}$ we get that 
\begin{align*}
    \textrm{exp}|_{N^{\varepsilon}M}:N^{\varepsilon}M \longrightarrow \overline{M}
\end{align*}
is an immersion and a local embedding. Notice that for every $\tilde\varepsilon \leq \varepsilon$ also $\textrm{exp}|_{N^{\tilde\varepsilon}M}$ is an immersion and a local embedding. We claim that there exists an $\tilde\varepsilon < \varepsilon$ such that this restriction is also globally injective. If this is the case, then the restriction to the closure of the $\frac{\tilde\varepsilon}{2}$--small normal bundle is an embedding, since injective immersions with compact domain are embeddings. It follows that any number less than $\frac{\tilde\varepsilon}{2}$ satisfies the statement of the theorem. \\
To prove the claim we argue by contradiction: suppose that for every $n \in \N$ there exist $(x_n,v_n)$, $(y_n,w_n) \in N^{\frac{1}{n}}M$ such that $\textrm{exp}(x_n,v_n)=\textrm{exp}(y_n,w_n)$. Since $M$ is compact, up to restricting to subsequences we can assume that $x_n$ converges to $\overline x \in M$ and $y_n$ converges to $\overline y \in M$, while $v_n$ and $w_n$ both converge to $0$ since $v_n,w_n \in B(0,\frac{1}{n})$ for every $n \in N$. By compactness of $M$ there exists $\delta >0$ such that for every $p \in M$ the map $\textrm{exp}_p:B(0,\delta) \subset T_p\overline{M}\longrightarrow \overline{M}$ is a diffeomorphism on its image, where $\textrm{exp}_p(z)=\textrm{exp}(p,z)$. It follows that since $v_n \longrightarrow 0$, for $n$ large enough $\gamma_{(x_n,v_n)}$ will be the unique geodesic joining $x_n$ with $\textrm{exp}_{x_n}(v_n)=\textrm{exp}(x_n,v_n)=\gamma_{(x_n,v_n)}(1)$ and $d_g(x_n,\textrm{exp}(x_n,v_n))=\norm{v_n}$. Analogously, for $n$ large enough we will also have $d_g(y_n,\textrm{exp}(y_n,w_n))=\norm{w_n}$, where we stress that the uniformity of $\delta$ is crucial. Since by hypothesis $\textrm{exp}(x_n,v_n)=\textrm{exp}(y_n,w_n)$, we have that 
\begin{align*}
    d_g(x_n,y_n)\leq d_g(x_n,\mathrm{exp}(x_n,v_n))+d_g(y_n,\textrm{exp}(y_n,w_n))=\norm{v_n}+\norm{w_n}\longrightarrow 0,
\end{align*}
and this forces $\overline x=\overline y$. Then for $n$ sufficiently large, we have that $(x_n,v_n)$, $(y_n,w_n) \in W_p$ for some $p \in M$, but on every $W_p$ we have a local embedding, leading to a contradiction. The proof is concluded.
\medskip 

\section{}\label{integrationappendix}
\subsection*{Proof of lemma \eqref{integrationlemma}}
We will use the same notations as in section \ref{sectionlemma}. We want to prove that 
\begin{align*}
    \varphi^*(\mathrm{vol}_{S^m}) = \bigl(1-\abs{z}^2\bigr)^{\frac{k-1}{2}}\ \mathrm{vol}_{S^k}\wedge \mathrm{vol}_{\overset{\circ}{D}{} ^{m-k}},
\end{align*}
where $\varphi$ is given by \eqref{parametrizationlemma}. By the usual formula for the pullback of a differential form through a diffeomorphism, we have 
\begin{align*}
    \varphi^*(\mathrm{vol}_{S^m})=\abs{\det(J\varphi^t \cdot J\varphi))}^{\frac{1}{2}}\ \mathrm{vol}_{\overset{\circ}{D}{} ^{k}}\wedge \mathrm{vol}_{\overset{\circ}{D}{} ^{m-k}},
\end{align*}
where $J\varphi$ denotes the $(m+1)\times m$ Jacobian matrix of $\varphi$ and $J\varphi^t$ is its transpose. Denote by $J\iota$ the Jacobian matrix of the inclusion $\iota:S^k \hookrightarrow \R^{k+1}$. Then $J\varphi$ is the following block matrix
\begin{align*}
    J\varphi(\sigma,z) = \left(\begin{array}{c|c}
     \sqrt{1-\abs{z}^2}J\iota(\sigma) & \biggl(\frac{-z_j}{\sqrt{1-\abs{z}^2}}\iota(\sigma)\biggr) \\ \hline
      0 & I_{m-k} 
\end{array}\right),
\end{align*}
where $I_{m-k}$ denotes the $(m-k)\times (m-k)$ identity matrix. Since $J\iota^t(\sigma) \cdot \iota(\sigma) =  \iota(\sigma) \cdot J\iota(\sigma)  =  0$ and $\iota(\sigma)^t \cdot \iota(\sigma)  =  1$, we obtain
\begin{align*}
    (J\varphi^t \cdot J\varphi)(\sigma,z) = 
    \left( \begin{array}{c|c}
        (1-\abs{z}^2)(J\iota^t \cdot J\iota)(\sigma) & 0 \\ \hline
         0 & \bigl(I_{m-k} + \frac{z \cdot z^t}{1-\abs{z}^2}\bigr) 
    \end{array}\right).
\end{align*}
For every $z \in \R^{m-k}$ consider $R \in O(m-k)$ such that $z=Re_1\abs{z}$, where $e_1=(1,0,\dots,0)$. Then we can compute the determinant of the lower right block as
\begin{align*}
    \det\biggl(I_{m-k}+\frac{z \cdot z^t}{1-\abs{z}^2}\biggr)  =  \det R\biggl(I_{m-k}+\frac{\abs{z}^2}{1-\abs{z}^2}E_{11}\biggr)R^t \ = \ \frac{1}{1-\abs{z}^2},
\end{align*}
where $E_{11}=e_1e_1^t$ has all zero entries except for the $(1,1)$--th one which is $1$. Recalling that the determinant of a diagonal block matrix is given by the product of the determinants of its blocks, we find the following expression
\begin{align*}
    \abs{\det(J\varphi^t \cdot J\varphi)} \ = \ \bigl(1-\abs{z}^2\bigr)^k\ \abs{\det(J\iota^t \cdot J\iota)}\ \frac{1}{1-\abs{z}^2} = \bigl(1-\abs{z}^2\bigr)^{k-1} \ \abs{\det(J\iota^t \cdot J\iota)}.
\end{align*}
Finally, applying again the formula for the pullback of a differential form, we can conclude that 
\begin{align*}
    \varphi^*(\mathrm{vol}_{S^m})=\bigl(1-\abs{z}^2\bigr)^{\frac{k-1}{2}} \ \abs{\det(J\iota^t \cdot J\iota)}^{\frac{1}{2}}\ \mathrm{vol}_{\overset{\circ}{D}{} ^{k}}\wedge \mathrm{vol}_{\overset{\circ}{D}{} ^{m-k}} =  \bigl(1-\abs{z}^2\bigr)^{\frac{k-1}{2}}\ \mathrm{vol}_{S^k} \wedge \mathrm{vol}_{\overset{\circ}{D}{} ^{m-k}}. 
\end{align*}
\medskip

\section{}
\subsection*{Proof of formula \eqref{explicitthirdpiece}}\label{expectationappendix}
We will use the same notations as Section \ref{volumechapter}. By linearity of the expectation, in order to prove formula \eqref{explicitthirdpiece} all we have to do is computing
\begin{align*}
    \underset{Q\in \mathrm{GOE}(n)}{\mathds{E}}[\det(I_n-\lambda Q)]=\sum_{j=0}^n(-1)^j\lambda^j\underset{Q\in \mathrm{GOE}(n)}{\mathds{E}}[g_j(Q)],
\end{align*}
since the expectation of $g_j(Q)$ can then be deduced by looking at the degree $j$ coefficient in above polynomial expression in $\lambda$. 

This can be computed immediately using \cite[Eq. (2.2.38)]{mehta}, but we also give a simple derivation of this computation for the sake of completeness.

First, we write the determinant according to its very definition
\begin{align*}
     \det(I_n-\lambda Q)=\sum_{\sigma \in S_n}{\mathrm{sgn}(\sigma)\prod_{i=1}^n{\bigl(\delta_{i\sigma(i)}-\lambda Q_{i\sigma(i)}\bigr)}},
 \end{align*}
 where $S_n$ is the group of permutations on $\{1,\dots,n\}$ and $\mathrm{sgn}(\sigma)$ is the signature of the permutation $\sigma \in S_n$. Recall that for $Q \in \mathrm{GOE}(n)$ we have $Q_{ii}\sim N(0,1)$ and $Q_{ij}\sim N(0,\frac{1}{2})$ for $i\neq j$ and, apart from the obvious symmetry conditions, the entries are independent.

 By linearity
 \begin{align}\label{expectationcomputation}
    \underset{Q\in \mathrm{GOE}(n)}{\mathds{E}}[\det(I_n-\lambda Q)] \ = \ \sum_{\sigma \in S_n}{\mathrm{sgn}(\sigma)\underset{Q\in \mathrm{GOE}(n)}{\mathds{E}}\biggl[\ \prod_{i=1}^n{\bigl(\delta_{i\sigma(i)}-\lambda Q_{i\sigma(i)}\bigr)}\biggr]}.
\end{align}
Given $\sigma \in S_n$, suppose that it contains a cycle of length at least $3$, i.e. there exists $i\in\{1,\dots,n\}$ such that $\sigma(i) \neq i$ and $\sigma^2(i)\neq i$. Then, by independence of the entries, in the term of \eqref{expectationcomputation} corresponding to $\sigma$, we can split the expectation into a product of expectations, separating the term corresponding to such $i$. Since $\delta_{i\sigma(i)}=0$ and $Q_{i\sigma(i)}$ is centered, this expectation is $0$ and $\sigma$ gives no contribution to \eqref{expectationcomputation}. It follows that the only permutations contributing to \eqref{expectationcomputation} are those formed by transpositions and fixed points only. \\
For such a $\sigma \in S_n$, denote by $\mathrm{fix}(\sigma)=\{i \in \{1,\dots,n\} \ | \ \sigma(i)=i\}$ the set of fixed points of $\sigma$ and by $s(\sigma)$ the number of disjoint transpositions in $\sigma$. Then we have 
\begin{align*}
    \underset{Q \in \mathrm{GOE}(n)}{\mathds{E}}\biggl[\prod_{i=1}^n{\bigl(\delta_{i\sigma(i)}-\lambda Q_{i\sigma(i)}\bigr)}\biggr] \ & = \ \mathds{E}\biggl[\prod_{i\in \mathrm{fix}(\sigma)}{\bigl(1-\lambda\xi_i\bigr)}\biggr]\cdot \mathds{E}\biggl[\prod_{k=1}^{s(\sigma)}{\frac{1}{2}\lambda^2\gamma_k^2}\biggr] \ = \\ & = \ \biggl(\,\prod_{i\in \mathrm{fix}(\sigma)}{\mathds{E}\bigl[\bigl(1-\lambda \xi_i\bigr)\bigr]}\biggr) \cdot \biggl(\prod_{k=1}^{s(\sigma)}{\mathds{E}\biggl[\frac{1}{2}\lambda^2 \gamma_k^2\biggr]}\biggr),
\end{align*}
where $\xi_i \sim N(0,1)$ and $\gamma_k \sim N(0,1)$. For these terms we have 
\begin{align*}
    &\mathds{E}\bigl[\bigl(1-\lambda\xi_i\bigr)\bigr] = 1, \\
    &\mathds{E}\bigl[\frac{1}{2}\lambda^2 \gamma_k^2\bigr]  = \frac{1}{2}\lambda^2 \mathds{E}\bigl[\gamma_k^2\bigr]  = \frac{1}{2}\lambda^2.
\end{align*}
The contribution of such $\sigma \in S_n$ in \eqref{expectationcomputation} is
\begin{align}\label{sigmacontribution}
    \mathrm{sgn(\sigma)}\underset{Q \in \mathrm{GOE(n)}}{\mathds{E}}\biggl[\prod_{i=1}^n{\bigl(\delta_{i\sigma(i)}-\lambda Q_{i\sigma(i)}\bigr)}\biggr] \ = \ \mathrm{sgn}(\sigma) \biggl(\frac{1}{2}\lambda^2\biggr)^{s(\sigma)},
\end{align}
and notice it depends only on $s(\sigma)$. To conclude the computation we, therefore, have to count how many permutations in $S_n$ are given by exactly $k$ disjoint transpositions for every $k=0,\dots,\lfloor \frac{n}{2}\rfloor$. Denote this number by $N(k)$. To construct a permutation with exactly $k$ disjoint transpositions we proceed as follows: choose two elements in $\{1,\dots,n\}$ forming the first transposition, then choose another $2$ among the remaining ones to form the second transposition and so on until the $k$--th one is formed. Moreover, since the supports of the transpositions are disjoint, the order in which they are picked is not relevant. It follows that 
\begin{align*}
    N(k)=\frac{\binom{n}{2}\binom{n-2}{2}\dots\binom{n-2k+2}{2}}{k!}=\frac{n!}{2^k(n-2k)!\ k!}
\end{align*}
and using this and \eqref{sigmacontribution} in \eqref{expectationcomputation} gives
\begin{align}\label{expectationfinal}
     \underset{Q\in \mathrm{GOE}(n)}{\mathds{E}}[\det(I_n-\tau Q)] \ = \ \sum_{k=0}^{\lfloor\frac{n}{2}\rfloor}{(-1)^kN(k)\biggl(\frac{1}{2}\lambda^2\biggr)^k} \ = \ \sum_{k=0}^{\lfloor\frac{n}{2}\rfloor}{(-1)^k\lambda^{2k}\frac{(2k)!}{2^{2k}k!}\binom{n}{2k}}.
\end{align}
Since the expectation of $g_j(Q)$ is given by the degree $j$ term in \eqref{expectationfinal} multiplied by $(-1)^j$, we obtain
\begin{align}\label{gjexpectation}
   \underset{Q \in \mathrm{GOE}(n)}{\mathds{E}}\bigl[g_j(Q)\bigr]= \begin{cases}
    0 & \text{if} \ j \ \text{odd} \\ \\
    \frac{(-1)^{\frac{j}{2}}}{2^j}\frac{j!}{(\frac{j}{2})!}\binom{n}{j} & \text{if} \ 0\leq j \leq n, \ j \ \text{even}
    \end{cases}.
\end{align}
Plugging \eqref{gjexpectation} into \eqref{g_jbeforesub}, we finally get \eqref{explicitthirdpiece}.

\bibliographystyle{alpha}
\bibliography{bibliototale}

\newcommand{\etalchar}[1]{$^{#1}$}
\begin{thebibliography}{SDLF{\etalchar{+}}17}

\bibitem[AGH{\etalchar{+}}14]{tensorslatentvariable}
Animashree Anandkumar, Rong Ge, Daniel Hsu, Sham~M. Kakade, and Matus
  Telgarsky.
\newblock Tensor decompositions for learning latent variable models.
\newblock {\em J. Mach. Learn. Res.}, 15:2773--2832, 2014.

\bibitem[AGZ10]{Zeitounirandom}
Greg~W. Anderson, Alice Guionnet, and Ofer Zeitouni.
\newblock {\em An introduction to random matrices}, volume 118 of {\em
  Cambridge Studies in Advanced Mathematics}.
\newblock Cambridge University Press, Cambridge, 2010.

\bibitem[ALLF07]{diffusionbrain}
Andrew Alexander, Jee Lee, Mariana Lazar, and Aaron Field.
\newblock Diffusion tensor imaging of the brain.
\newblock {\em Neurotherapeutics : the journal of the American Society for
  Experimental NeuroTherapeutics}, 4:316--29, 08 2007.

\bibitem[BC13]{Burgissercondition}
Peter B\"{u}rgisser and Felipe Cucker.
\newblock {\em Condition}, volume 349 of {\em Grundlehren der Mathematischen
  Wissenschaften [Fundamental Principles of Mathematical Sciences]}.
\newblock Springer, Heidelberg, 2013.
\newblock The geometry of numerical algorithms.

\bibitem[BL22]{BL}
Saugata Basu and Antonio Lerario.
\newblock Hausdorff approximations and volume of tubes of singular algebraic
  sets.
\newblock {\em Mathematische Annalen}, August 2022.

\bibitem[Bre19]{Breidingeigenvalues}
Paul Breiding.
\newblock How many eigenvalues of a random symmetric tensor are real?
\newblock {\em Trans. Amer. Math. Soc.}, 372(11):7857--7887, 2019.

\bibitem[Bue06]{Burgisseraverage}
Peter Buergisser.
\newblock Average volume, curvatures, and euler characteristic of random real
  algebraic varieties, 2006.

\bibitem[CdS01]{Cannassymplectic}
Ana Cannas~da Silva.
\newblock {\em Lectures on symplectic geometry}, volume 1764 of {\em Lecture
  Notes in Mathematics}.
\newblock Springer-Verlag, Berlin, 2001.

\bibitem[CGO14]{secantlectures}
Enrico Carlini, Nathan Grieve, and Luke Oeding.
\newblock Four lectures on secant varieties.
\newblock In {\em Connections between algebra, combinatorics, and geometry},
  volume~76 of {\em Springer Proc. Math. Stat.}, pages 101--146. Springer, New
  York, 2014.

\bibitem[dC92]{doCarmo}
Manfredo Perdig\~{a}o do~Carmo.
\newblock {\em Riemannian geometry}.
\newblock Mathematics: Theory \& Applications. Birkh\"{a}user Boston, Inc.,
  Boston, MA, 1992.
\newblock Translated from the second Portuguese edition by Francis Flaherty.

\bibitem[DH16]{DHrankone}
Jan Draisma and Emil Horobe\c{t}.
\newblock The average number of critical rank-one approximations to a tensor.
\newblock {\em Linear Multilinear Algebra}, 64(12):2498--2518, 2016.

\bibitem[DREG22]{dirocco2}
Sandra Di~Rocco, David Eklund, and Oliver G\"{a}fvert.
\newblock Sampling and homology via bottlenecks.
\newblock {\em Math. Comp.}, 91(338):2969--2995, 2022.

\bibitem[DREW20]{dirocco1}
Sandra Di~Rocco, David Eklund, and Madeleine Weinstein.
\newblock The bottleneck degree of algebraic varieties.
\newblock {\em SIAM J. Appl. Algebra Geom.}, 4(1):227--253, 2020.

\bibitem[EK95]{EdelmanKostlan95}
Alan Edelman and Eric Kostlan.
\newblock How many zeros of a random polynomial are real?
\newblock {\em Bull. Amer. Math. Soc. (N.S.)}, 32(1):1--37, 1995.

\bibitem[Fri13]{friedland}
Shmuel Friedland.
\newblock Best rank one approximation of real symmetric tensors can be chosen
  symmetric.
\newblock {\em Front. Math. China}, 8(1):19--40, 2013.

\bibitem[GFE09]{mostquantum}
D.~Gross, S.~T. Flammia, and J.~Eisert.
\newblock Most quantum states are too entangled to be useful as computational
  resources.
\newblock {\em Phys. Rev. Lett.}, 102(19):190501, 4, 2009.

\bibitem[Gra04]{Graytubes}
Alfred Gray.
\newblock {\em Tubes}, volume 221 of {\em Progress in Mathematics}.
\newblock Birkh\"{a}user Verlag, Basel, second edition, 2004.
\newblock With a preface by Vicente Miquel.

\bibitem[HKW{\etalchar{+}}09]{entanglementsymmetric}
Robert H\"{u}bener, Matthias Kleinmann, Tzu-Chieh Wei, Carlos
  Gonz\'{a}lez-Guill\'{e}n, and Otfried G\"{u}hne.
\newblock Geometric measure of entanglement for symmetric states.
\newblock {\em Phys. Rev. A (3)}, 80(3):032324, 5, 2009.

\bibitem[How93]{Howardkinematic}
Ralph Howard.
\newblock The kinematic formula in {R}iemannian homogeneous spaces.
\newblock {\em Mem. Amer. Math. Soc.}, 106(509):vi+69, 1993.

\bibitem[HQZ16]{HuQiZhang}
Shenglong Hu, Liqun Qi, and Guofeng Zhang.
\newblock Computing the geometric measure of entanglement of multipartite pure
  states by means of non-negative tensors.
\newblock {\em Phys. Rev. A}, 93(1):012304, 7, 2016.

\bibitem[IN66]{Itzykson}
C.~Itzykson and M.~Nauenberg.
\newblock Unitary groups: {R}epresentations and decompositions.
\newblock {\em Rev. Modern Phys.}, 38:95--120, 1966.

\bibitem[Kro08]{Kroonenberg}
Pieter~M. Kroonenberg.
\newblock {\em Applied multiway data analysis}.
\newblock Wiley Series in Probability and Statistics. Wiley-Interscience [John
  Wiley \& Sons], Hoboken, NJ, 2008.
\newblock With a foreword by Willem J. Heiser and Jarqueline Meulman.

\bibitem[Lan12]{Landsbergtensors}
J.~M. Landsberg.
\newblock {\em Tensors: geometry and applications}, volume 128 of {\em Graduate
  Studies in Mathematics}.
\newblock American Mathematical Society, Providence, RI, 2012.

\bibitem[McC87]{McCullagh}
Peter McCullagh.
\newblock {\em Tensor methods in statistics}.
\newblock Monographs on Statistics and Applied Probability. Chapman \& Hall,
  London, 1987.

\bibitem[Meh91]{mehta}
Madan~Lal Mehta.
\newblock {\em Random matrices}.
\newblock Academic Press, Boston, New York, San Diego, 1991.

\bibitem[Nij74]{Nijenhuischern}
Albert Nijenhuis.
\newblock On {C}hern's kinematic formula in integral geometry.
\newblock {\em J. Differential Geometry}, 9:475--482, 1974.

\bibitem[QL17]{Qitensoranalysis}
Liqun Qi and Ziyan Luo.
\newblock {\em Tensor analysis}.
\newblock Society for Industrial and Applied Mathematics, Philadelphia, PA,
  2017.
\newblock Spectral theory and special tensors.

\bibitem[Sak16]{Sakatasurvey}
Toshio Sakata.
\newblock {\em Applied Matrix and Tensor Variate Data Analysis}.
\newblock 01 2016.

\bibitem[SBG04]{tensorschemical}
A.~Smilde, Rasmus Bro, and P.~Geladi.
\newblock Multi way analysis — applications in chemical sciences.
\newblock 01 2004.

\bibitem[SDLF{\etalchar{+}}17]{tensorsignal}
Nicholas~D. Sidiropoulos, Lieven De~Lathauwer, Xiao Fu, Kejun Huang,
  Evangelos~E. Papalexakis, and Christos Faloutsos.
\newblock Tensor decomposition for signal processing and machine learning.
\newblock {\em IEEE Trans. Signal Process.}, 65(13):3551--3582, 2017.

\bibitem[Shi95]{Shimony}
Abner Shimony.
\newblock Degree of entanglement.
\newblock In {\em Fundamental problems in quantum theory ({B}altimore, {MD},
  1994)}, volume 755 of {\em Ann. New York Acad. Sci.}, pages 675--679. New
  York Acad. Sci., New York, 1995.

\bibitem[SS93a]{CB2}
M.~Shub and S.~Smale.
\newblock Complexity of {B}ezout's theorem. {II}. {V}olumes and probabilities.
\newblock In {\em Computational algebraic geometry ({N}ice, 1992)}, volume 109
  of {\em Progr. Math.}, pages 267--285. Birkh\"{a}user Boston, Boston, MA,
  1993.

\bibitem[SS93b]{CB1}
Michael Shub and Steve Smale.
\newblock Complexity of {B}\'{e}zout's theorem. {I}. {G}eometric aspects.
\newblock {\em J. Amer. Math. Soc.}, 6(2):459--501, 1993.

\bibitem[SS93c]{CB3}
Michael Shub and Steve Smale.
\newblock Complexity of {B}ezout's theorem. {III}. {C}ondition number and
  packing.
\newblock volume~9, pages 4--14. 1993.
\newblock Festschrift for Joseph F. Traub, Part I.

\bibitem[Wey39]{Weyltubes}
Hermann Weyl.
\newblock On the {V}olume of {T}ubes.
\newblock {\em Amer. J. Math.}, 61(2):461--472, 1939.

\bibitem[ZG01]{ZhangGRQ}
Tong Zhang and Gene~H. Golub.
\newblock Rank-one approximation to high order tensors.
\newblock {\em SIAM J. Matrix Anal. Appl.}, 23(2):534--550, 2001.

\end{thebibliography}

\end{document}